\documentclass[reqno,final]{siamart0216}
\usepackage{amssymb,amsmath,amsfonts,amssymb}
\usepackage{graphicx}
\usepackage{graphics,color}
\usepackage{bmpsize}

\newcommand{\D}{\mathrm{D}} 

\usepackage{mathrsfs}                    
\renewcommand{\d}{\mathrm{d}}            

\usepackage[normalem]{ulem} %

\newcommand\rvec{{\bf r}}

\newcommand\nvec{{\bf n}}
\newcommand\xvec{{\bf x}}

\newcommand\pvec{{\bf p}}
\newcommand\qvec{{\bf q}}

\newcommand\Qvec{{\bf Q}}

\newcommand{\Rr}{{\mathbb R}} 

\newcommand\dd{\mathrm{d}}
\newcommand\pp{\partial}
\newcommand\x{\mathbf{x}}
\newcommand\e{\mathbf{e}}

\usepackage[utf8x]{inputenc}
\usepackage[english]{babel}
\usepackage{overpic}

\newcommand{\Ivec}{\mathbf{I}}           
\DeclareMathOperator{\tr}{tr}            
\newcommand{\abs}[1]{\left| #1 \right|}  

\title{Order Reconstruction for Nematics on Squares with Isotropic Inclusions: A Landau-de Gennes Study}
\author{Yiwei Wang, Giacomo Canevari \& Apala Majumdar}

\begin{document}
\maketitle
\begin{abstract}
We prove the existence of a well order reconstruction solution (WORS) type Landau-de Gennes critical point on a square domain with an isotropic concentric square inclusion, with tangent boundary conditions on the outer square edges. There are two geometrical parameters --- the outer square edge length $\lambda$, and the aspect ratio $\rho$, which is the ratio of the inner and outer square edge lengths. The WORS exists for all geometrical parameters and for all temperatures; we prove that the WORS is globally stable for either $\lambda$ small enough or for $\rho$ sufficiently close to unity. We 
study three different types of critical points in this model setting: critical points with the minimal two degrees of freedom consistent with the imposed boundary conditions, critical points with three degrees of freedom and critical points with five degrees of freedom. In the two-dimensional case, we use $\Gamma$-convergence techniques to identify the energy-minimizing competitors. We decompose the second variation of the Landau-de Gennes energy into three separate components to study the effects of different types of perturbations on the WORS solution and find that it is most susceptible to in-plane perturbations. In the three-dimensional setting, we numerically find up to $28$ critical points for moderately large values of $\rho$ and we find two critical points with the full five degrees of freedom for very small values of $\rho$, with an escaped profile around the isotropic square inclusion.
\end{abstract}

\section{Introduction}
\label{sec:intro}

Nematic liquid crystals (NLCs) are classical examples of partially ordered materials that combine the fluidity of liquids with a degree of long-range orientational order \cite{dg,virga}. There is substantial interest in pattern formation for NLCs in confinement, of which NLC-filled square chambers are popular examples \cite{tsakonas, lewissoftmatter, cleaver, luo2012}. This paper focuses on stable NLC configurations for square geometries with a square hole, referred to as an isotropic inclusion which locally destroys the surrounding nematic ordering. Such holes can be created by laser treatments or e-beam lithography techniques~\cite{e-beam} 
and domains with inclusions offer new possibilities for exotic pattern formation.

This paper is primarily motivated by the numerical results in \cite{kraljmajumdar} and the analytical results in \cite{majcanevarispicer2017}, both within the celebrated Landau-de Gennes (LdG) theory for nematic liquid crystals. The LdG theory describes the nematic state by a macroscopic order parameter, the $\mathbf{Q}$-tensor order parameter which is, mathematically speaking, a symmetric traceless $3\times 3$ matrix. The eigenvectors of the $\mathbf{Q}$-tensor represent the physically preferred directions for molecular alignment or the directions of orientational ordering and the corresponding eigenvalues are a measure of the degree of the order about the eigenvectors \cite{ejam2010, dg, virga}. In \cite{kraljmajumdar}, the authors numerically discover a novel ``Well Order Reconstruction Solution'' (WORS) for square domains with tangent boundary conditions on the square edges. This WORS solution has a constant eigenframe, featured by a cross that connects the square diagonals such that the $\mathbf{Q}$-tensor has two degenerate positive eigenvalues and a distinct negative eigenvalue along the diagonal cross, referred to as negative uniaxiality which is a signature of nematic defects. The WORS is globally stable for small square domains, typically of the order of tens to hundreds of nanometers. In \cite{majcanevarispicer2017}, the authors analytically prove the existence of the WORS solution for all square sizes and at a special temperature, reduce the analysis of the WORS solution to a scalar variational problem. The authors prove the global stability of the WORS solution for small square domains, the instability of the WORS solution for larger domains and prove that the WORS solution branch undergoes a supercritical pitchfork bifurcation as the square size increases, in the reduced scalar setting. The supercritical pitchfork bifurcation result is quite specific to the scalar problem and it is not clear if it holds for the full LdG problem with five degrees of freedom. 

In \cite{kraljmajumdar}, the authors numerically study the effects of square inclusions or square holes on the stability and properties of the WORS on square domains. For concentric square inclusions (i.e. square holes that have the same centre as the square domain), the WORS exists although the stability properties depend on both the square size and the domain aspect ratio (the ratio of the inclusion size to the domain size). For an off-centered square inclusion, we lose the distinctive diagonal cross and the regions of negative uniaxiality become localised near the square edges. 
In this paper, we study square domains with concentric isotropic square inclusions. Mathematically, we study a  boundary value problem for the LdG $\mathbf{Q}$-tensor on this domain, with $\mathbf{Q}=0$ on the inclusion boundary and with Dirichlet boundary conditions on the outer boundary consistent with the experimentally imposed tangent boundary conditions in \cite{tsakonas}. 
We prove the existence of a WORS-like solution for this model problem, with a constant eigenframe and a diagonal cross (along which the LdG $\mathbf{Q}$-tensor has two equal eigenvalues) that connects the vertices of the inner and outer squares. This existence theorem is true for all square sizes and aspect ratios (ratio of the inner square size to the outer square size) and for all temperatures. Following the arguments in \cite{lamy2014} and \cite{majcanevarispicer2017}, we can also prove that the WORS is globally stable, i.e. is the global minimizer of the LdG energy for this model problem, for either squares that are sufficiently small or for aspect ratios sufficiently close to unity. In this sense, we provide some theoretical foundations for the numerical results in \cite{kraljmajumdar}. 

The analysis of the WORS is inherently two-dimensional in the presence of a square inclusion by contrast with the framework in \cite{majcanevarispicer2017} where the authors could study a scalar variational problem at a special temperature. We have conflicting boundary conditions on the inner and outer squares and we need to exploit two out of the five degrees of freedom of the LdG $\mathbf{Q}$-tensor, to describe the WORS for all temperatures. We perform a $\Gamma$-convergence analysis of a reduced LdG energy, in terms of these two degrees of freedom, to deduce qualitative properties of energy minimizers in this two-dimensional setting, in the limit of the square size $\lambda \to \infty$. We are able to identify at least three competing configurations in the reduced two-dimensional setting: the WORS configuration, a BD (boundary distortion)-configuration with a pair of distinctive edge transition layers along which the LdG $\mathbf{Q}$-tensor transitions between two distinct states and an ESC-configuration around which the nematic molecules escape into the third dimension around the isotropic inclusion. We compute specific minimality criteria of the WORS in terms of the material constants, the temperature and the geometric aspect ratio, in this asymptotic limit.

The $\Gamma$-convergence analysis is complemented by a detailed numerical study of the critical points of the LdG energy for this model problem, using finite-difference based numerical methods and deflation techniques \cite{farrell2015deflation}. Numerical investigations show that the ESC-configuration cannot have lower LdG energy than the WORS or BD configurations, which is also corroborated by the minimality estimates for the WORS, BD and ESC-configurations yielded by the $\Gamma$-convergence analysis in the $\lambda \to \infty$ limit. Hence, we restrict ourselves to a detailed study of the stability of the WORS and BD configurations, both of which have constant eigenframes and have distinct defect lines or transition layers. In the case of the WORS, the transition layers are supported along the diagonals and for the BD solution, along a pair of opposite square edges. We study the second variation of the LdG energy and decompose the second variation into three components --- the second variation in the two-dimensional class of perturbations that do not distort the constant eigenframes of the WORS and BD configurations, the second variation with respect to in-plane perturbations of the eigenframe and the second variation with respect to out-of-plane perturbations of the eigenframe. We believe that this decomposition will be useful for stability analysis of general critical points for more general model problems. We numerically test the stabilities of the WORS and the BD-configurations with respect to the three different kinds of perturbations at the special temperature employed in \cite{majcanevarispicer2017}, primarily to reduce the number of variables in the problem and this temperature is a special reference point. As expected, we find that the WORS is globally stable with respect to all perturbations for aspect ratios that are sufficiently close to unity i.e. narrow square annuli. Both the WORS and BD configurations are stable with respect to out-of-plane perturbations. It is interesting that the BD-configuration is always unstable with respect to in-plane perturbations i.e. the BD-configuration is never a stable critical point of the LdG energy for this model problem.

We briefly comment on how these results relate to the numerical results in \cite{martinrobinson2017} where the authors study the solution landscape as a function of the square size in a three-dimensional LdG framework, neglecting the out-of-plane components. They numerically find that the WORS solution branch which is globally stable for small squares and loses stability as the square size increases. The WORS solution loses stability with respect to BD-like configurations in the restricted two-dimensional class of perturbations which preserve the constant eigenframe but these BD-configurations are unstable in the class of perturbations which allow for in-plane distortions of the eigenframe. Indeed, the authors numerically observe at least four bifurcating solution branches from the WORS solution branch --- two unstable BD solution branches and two stable diagonal solution branches which do not have the constant eigenframe property. For larger squares, the BD solution branches connect to the familiar stable rotated solutions, which do not have a constant eigenframe, and for which the nematic molecules rotated by $\pi$ radians in the square plane, between a pair of parallel square edges. 

Finally, we comment on why the WORS and BD-configurations are stable with respect to all out-of-plane perturbations for this model problem. It is rigorously proven in \cite{sternberggolovaty2015} that for certain thin geometries (where the vertical dimension is much smaller than the lateral dimensions) and for certain surface energies consistent with tangent boundary conditions, the LdG energy minimization problem reduces to a variational problem on the two-dimensional cross-section (such as the square domain in our case) and energy minimizers indeed only have three degrees of freedom. The energy minimizers have a fixed eigenvector in the $\mathbf{z}$-direction; one degree of freedom describes the in-plane alignment of the NLC molecules and two scalar order parameters account for the in-plane ordering and the ordering about the $\mathbf{z}$-direction.
For the WORS and the BD-configurations, the in-plane alignment is fixed by the constant eigenframe and hence, they belong to a sub-class of this reduced three-dimensional setting. In light of the rigorous results in \cite{sternberggolovaty2015}, it is not surprising that the instabilities arise in the reduced three-dimensional setting.

The paper is organized as follows. In Section~\ref{sec:prelimiaries} and \ref{sec:reduction}, we set up the geometric domain and the problem definition, along with recalling the mathematical framework of the LdG theory and proving the existence and uniqueness theorems for the WORS. In Section~\ref{sec:gamma}, we perform the $\Gamma$-convergence analysis for the limit of large domains and in Section~\ref{sec:numerics}, we present and analyse our numerical results. In Section~\ref{sec:conclusion}, we briefly present our conclusions.
\section{Preliminaries}
\label{sec:prelimiaries}

We model nematic profiles on two-dimensional squares with an isotropic inclusion within the Landau-de Gennes (LdG) theory for nematic liquid crystals.
The LdG theory is one of the most powerful continuum theories for nematic liquid crystals and describes the nematic state by a macroscopic order parameter 
--- the LdG $\mathbf{Q}$-tensor that is a macroscopic measure of material anisotropy. 
The LdG $\Qvec$-tensor  is a symmetric traceless $3\times 3$ matrix i.e. 
$$\Qvec \in S_0 := \left\{ \Qvec\in \mathbb{M}^{3\times 3}\colon Q_{ij} = Q_{ji}, \ Q_{ii} = 0 \right\}.$$
A $\Qvec$-tensor is said to be (i) isotropic if~$\Qvec=0$, (ii) uniaxial if $\Qvec$
has a pair of degenerate non-zero eigenvalues and (iii) biaxial if~$\Qvec$ has three distinct eigenvalues~\cite{dg,newtonmottram}. A
uniaxial $\Qvec$-tensor can be written as $\Qvec_u = s \left(\nvec \otimes \nvec - \Ivec/3\right)$ with~$\Ivec$ the $3\times 3$ identity matrix,
$s\in\Rr$ and~$\nvec\in S^2$, a unit vector. The scalar, $s$, is an order parameter which measures the degree of orientational order.
The vector, $\nvec$, is the eigenvector with the non-degenerate eigenvalue, referred to as the ``director'' and labels the single distinguished direction of uniaxial nematic alignment~\cite{virga,dg}. 


We work with a simple form of the LdG energy given by
\begin{equation} \label{eq:2} 
 I[\Qvec] :=  \int_{\Omega} \frac{L}{2} \left|\nabla\Qvec \right|^2 + f_B(\Qvec) \, \mathrm{d}A,
\end{equation}
where $\Omega\subseteq\Rr^2$ is a two-dimensional domain,
\begin{equation} \label{eq:3}
 |\nabla \Qvec |^2 := \frac{\partial Q_{ij}}{\partial r_k}\frac{\partial Q_{ij}}{\partial r_k},
 \qquad f_B(\Qvec) := \frac{A}{2} \tr\Qvec^2 - \frac{B}{3} \tr\Qvec^3 + \frac{C}{4}\left(\tr\Qvec^2 \right)^2.
\end{equation}

The variable~$A = \alpha (T - T^*)$ is the re-scaled temperature, $\alpha$, $L$, $B$, $C>0$ are material-dependent constants
and~$T^*$ is the characteristic nematic supercooling temperature~\cite{dg,newtonmottram}.
Further~$\rvec:=(x, \, y)$, $\tr\Qvec^2 = Q_{ij}Q_{ij}$ and $\tr\Qvec^3 = Q_{ij} Q_{jk}Q_{ki}$ for $i$, $j$, $k = 1, \, 2, \, 3$.
It is well-known that all stationary points of the thermotropic
potential, $f_B$, are either uniaxial or isotropic~\cite{dg,newtonmottram,ejam2010}.
The re-scaled temperature~$A$ has three characteristic values: (i)~$A=0$, below which the isotropic phase $\Qvec=0$ loses stability, 
(ii) the nematic-isotropic transition temperature, $A={B^2}/{27 C}$, at which $f_B$ is minimized by the isotropic
phase and a continuum of uniaxial states with $s=s_+ ={B}/{3C}$ and
$\nvec$ arbitrary, and (iii) the nematic supercooling temperature, $A = {B^2}/{24 C}$, above which the isotropic state is the unique critical point of $f_B$.

We work with $A<0$ i.e. low temperatures and the numerical work in this paper focuses on a special temperature,
$A=-{B^2}/{3C}$, largely to facilitate comparison with~\cite{majcanevarispicer2017}.
Our analytical results are true for all temperatures, $A<0$. For a given $A<0$, let
$\mathscr{N} := \left\{ \Qvec \in S_0\colon \Qvec = s_+ \left(\nvec\otimes \nvec - \Ivec/3 \right) \right\}$ 
denote the set of minimizers of the bulk potential, $f_B$, with
\[
 s_+ := \frac{B + \sqrt{B^2 + 24|A| C}}{4C}
\]
and~$\nvec \in S^2$ arbitrary. In particular, this set is relevant to our choice of Dirichlet conditions for boundary-value problems.

We non-dimensionalize the system using a change of variables, $\bar{\rvec} = \rvec/ \lambda$,
where $\lambda$ is a characteristic length scale of the system.
The re-scaled LdG energy functional is then given by
\begin{equation} \label{eq:rescaled}
 \overline{I}[\Qvec] := \frac{I[\Qvec]}{L \lambda} = \int_{\overline{\Omega}}\frac{1}{2}\left| \overline{\nabla} \Qvec \right|^2 
 + \frac{\lambda^2}{L} f_B\left(\Qvec \right) \, \overline{\mathrm{d}A}.
\end{equation}
In~\eqref{eq:rescaled}, $\overline{\Omega}$ is the re-scaled domain, $\overline{\nabla}$ is the gradient with respect to
the re-scaled spatial coordinates and $\overline{\mathrm{d}A}$ is the re-scaled area element. The associated Euler-Lagrange equations are 
%
\begin{equation} \label{eq:6} 
 \bar{\Delta} \Qvec = \frac{\lambda^2}{L} \left\{ A\Qvec - B\left(\Qvec\Qvec - \frac{\Ivec}{3}|\Qvec|^2 \right)
 + C|\Qvec|^2 \Qvec \right\},
\end{equation}
where $(\Qvec\Qvec)_{ik} = Q_{ij}Q_{jk}$ with $i$, $j$, $k=1, \, 2, \, 3$. The system~\eqref{eq:6} comprises five coupled nonlinear
elliptic partial differential equations.
We treat $A$, $B$, $C$, $L$ as fixed constants and vary $\lambda$. 
In what follows, we drop the \emph{bars} and all statements are to be understood in terms of the re-scaled variables.


\section{The Variational Problem}
\label{sec:reduction}

\begin{figure}[t]
	\centering
 	\includegraphics[height=5 cm]{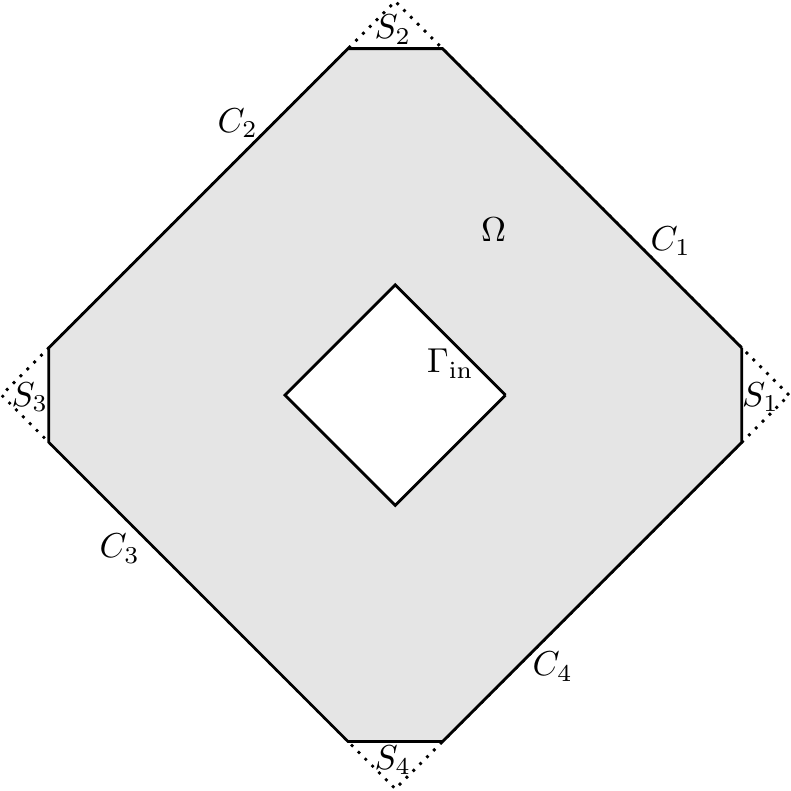} 
	\caption{The domain~$\Omega$.}
	\label{fig:domain}
\end{figure}

We take the rescaled domain~$\Omega\subseteq\Rr^2$ to be a truncated square with a
square inclusion. More precisely, for fixed~$0 < \rho < 1$, we define
\begin{equation} \label{eq:domain}
 \Omega := \left\{(x, \, y)\in\Rr^2\colon |x| < 1 - \varepsilon, \ |y| < 1 - \varepsilon,
 \ \rho < |x+y| < 1, \ \rho < |x-y| < 1 \right\} \! .
\end{equation}
The boundary, $\partial\Omega$, has two components, an inner boundary and an outer boundary.
The inner boundary, $\Gamma_{\mathrm{in}}$, is a square whose diagonals are parallel
to the coordinate axes, with side length~$\sqrt{2}\rho$.
The outer boundary, $\Gamma_{\mathrm{out}}$, consists of four ``long''
edges~$C_1, \, \ldots, \, C_4$, parallel to the lines~$y = x$ and~$y = -x$, and four
``short'' edges~$S_1, \, \ldots, \, S_4$, of length~$2\varepsilon$, parallel to the
$x$ and $y$-axes respectively. The long edges~$C_i$ are labeled counterclockwise
and $C_1$ is the edge contained in the first quadrant, i.e.
\[
 C_1 := \left\{(x, \, y)\in\Rr^2\colon x + y = 1, \ \varepsilon \leq x \leq 1 - \varepsilon \right\} \! .
\]
The short edges~$S_i$ 
are also labeled counterclockwise and
\[
 S_1 := \left\{(1 - \varepsilon, \, y)\in\Rr^2\colon |y|\leq \varepsilon\right\} \! .
\]
The domain is illustrated in Figure~\ref{fig:domain}.
We work with Dirichlet conditions on $\partial\Omega$. To mimic the isotropic inclusion, 
we impose \emph{isotropic} boundary conditions on the inner boundary~$\Gamma_{\mathrm{in}}$,
that is, we require
\begin{equation}\label{eq:bc0}
 \Qvec(\rvec) = \Qvec_{\mathrm{b}}(\rvec) := 0 \qquad \textrm{for } \rvec\in \Gamma_{\mathrm{in}}.
\end{equation}
We impose \emph{tangent} uniaxial Dirichlet conditions on the long edges, $C_1, \, \ldots, \, C_4$.
We fix $\Qvec = \Qvec_{\mathrm{b}}$ on $C_1, \, \ldots, \, C_4$ where
\begin{equation} \label{eq:bc1}
 \Qvec_{\mathrm{b}}(\rvec) := \begin{cases} 
  s_+\left( \nvec_1 \otimes \nvec_1 - \dfrac{\Ivec}{3} \right)  & \textrm{for } \rvec \in C_1 \cup C_3 \\
  s_+\left( \nvec_2 \otimes \nvec_2 - \dfrac{\Ivec}{3} \right)  & \textrm{for } \rvec \in C_2 \cup C_4;
 \end{cases}
\end{equation}
and
\[
 \nvec_1 := \frac{1}{\sqrt{2}}\left(-1, \, 1\right), \qquad  \nvec_2 := \frac{1}{\sqrt{2}}\left(1, \, 1 \right).
\]
The Dirichlet condition on the short edges is defined in terms of a 
function~$g\colon [-\varepsilon, \, \varepsilon]\to [-s_+/2, s_+/2]$. 
We assume that~$g$ is smooth (at least of class~$C^1$), odd
(i.e. $g(-s) = -g(s)$ for any~$s$), and satisfies $g(\varepsilon) = s_+/2$;
for instance, an admissible choice for~$g$ is
\[
 g(s) := \frac{s_+}{2\varepsilon} s \qquad 
 \textrm{for } -\varepsilon \leq s \leq\varepsilon.
\]
We fix $\Qvec = \Qvec_{\mathrm{b}}$ on $S_1, \, \ldots, \, S_4$ where
\begin{equation} \label{eq:bc2}
 \Qvec_{\mathrm{b}} := \begin{cases}
  g(y) \left(\nvec_1 \otimes \nvec_1 - \nvec_2\otimes \nvec_2 \right) - \dfrac{s_+}{6}\left(2 \hat{\mathbf{z}}\otimes\hat{\mathbf{z}}
  - \nvec_1\otimes \nvec_1 - \nvec_2\otimes \nvec_2 \right)  & \textrm{on  } S_1\cup S_3, \\
  g(x)\left(\nvec_1 \otimes \nvec_1 - \nvec_2\otimes \nvec_2 \right) - \dfrac{s_+}{6}\left(2 \hat{\mathbf{z}}\otimes\hat{\mathbf{z}}
  - \nvec_1\otimes \nvec_1 - \nvec_2\otimes \nvec_2 \right) & \textrm{on  } S_2\cup S_4.
 \end{cases}
\end{equation}
Given the Dirichlet conditions~\eqref{eq:bc0},~\eqref{eq:bc1} and~\eqref{eq:bc2},
we define our admissible space to be
\begin{equation} \label{eq:admissible}
 \mathscr{A} := \left\{ \Qvec \in W^{1,2}\left(\Omega, \, S_0 \right)\!\colon \Qvec = \Qvec_{\mathrm{b}}~\textrm{on} ~\partial \Omega \right\}.
\end{equation} 

We look for critical points of the re-scaled functional~\eqref{eq:rescaled} of the form
\begin{equation} \label{eq:d1}
 \begin{split}
  \Qvec(x, \, y) &= q_1(x, \, y) \left(\nvec_1 \otimes \nvec_1 - \nvec_2\otimes \nvec_2 \right) +
  q_3(x, \, y) \left(2 \hat{\mathbf{z}}\otimes\hat{\mathbf{z}} - \nvec_1\otimes \nvec_1 - \nvec_2\otimes \nvec_2 \right)
 \end{split}
\end{equation}
subject to the boundary conditions
\begin{equation} \label{eq:d2}
 q_1(x, \, y) = q_{1,\mathrm{b}} (x, \, y) :=
 \begin{cases}
  0           & \textrm{on  } \Gamma_{\mathrm{in}} \\
  {s_+}/{2}  & \textrm{on  } C_1\cup C_3 \\
 - {s_+}/{2}   & \textrm{on  } C_2\cup C_4 \\
  g(y)        & \textrm{on  } S_1\cup S_3 \\
  g(x)        & \textrm{on  } S_2\cup S_4;
 \end{cases}
\end{equation}
and
\begin{equation} \label{eq:d3}
 q_3(x, \, y) = q_{3,\mathrm{b}} (x, \, y) :=
 \begin{cases}
  0           & \textrm{on  } \Gamma_{\mathrm{in}} \\
  -{s_+}/{6}  & \textrm{on  } \Gamma_{\mathrm{out}}.
 \end{cases}
\end{equation}
For solutions of the form~\eqref{eq:d1}, the LdG Euler-Lagrange system (\ref{eq:6}) reduces to
\begin{equation} \label{eq:d4}
 \begin{aligned}
   \Delta q_1 &= \frac{\lambda^2}{L}\left\{A q_1 + 2B q_1 q_3 +
      2C\left(q_1^2 + 3q_3^2\right) q_1 \right\} \\
   \Delta q_3 &= \frac{\lambda^2}{L}\left\{A q_3 + B \left(\frac{1}{3}q_1^2 - q_3^2\right) +       
      2C\left(q_1^2 + 3q_3^2\right) q_3\right\} \! .
 \end{aligned}
\end{equation}
The partial differential equations \eqref{eq:d4} are precisely the Euler-Lagrange
equations associated with the functional
\begin{equation} \label{eq:J}
  J_\lambda[q_1, q_3]: = \int_{\Omega} \left( |\nabla q_1|^2 + 3 |\nabla q_3|^2
     + \frac{\lambda^2}{L} F(q_1,\, q_3)\right) \d A,
\end{equation}
where~$F$ is the polynomial potential given by
\begin{equation} \label{F}
 F(q_1, \,q_3) := A (q_1^2 + 3 q_3^2) + 2 B \, q_3 \,(q_1^2 - 2  q_3^2) 
  + C(q_1^2 + 3 q_3^2)^2 - F_{\min}
\end{equation}
and~$F_{\min} := As_+^2/3 - 2Bs_+^3/27 + Cs_+^4/9$ is a constant chosen so
that $\inf F=0$. By solving the criticality conditions~$\nabla_{(q_1, q_3)} F = 0$,
we find that~$F$ has exactly four critical points in the $(q_1, \, q_3)$-plane: 
the origin~$(0, \, 0)$, which is a local maximum,
and the points
\begin{equation} \label{critical_points}
 \pvec_1 := (-s_+/2,  \, -s_+/6), \qquad \pvec_2 := (s_+/2,  \, -s_+/6), 
 \qquad \pvec_3 := (0,  \, s_+/3),
\end{equation}
which are global minima. These critical points are illustrated in Figure~\ref{fig:criticalpoints}.

\begin{figure}[t]
	\centering 
 	\includegraphics[height=4.4cm]{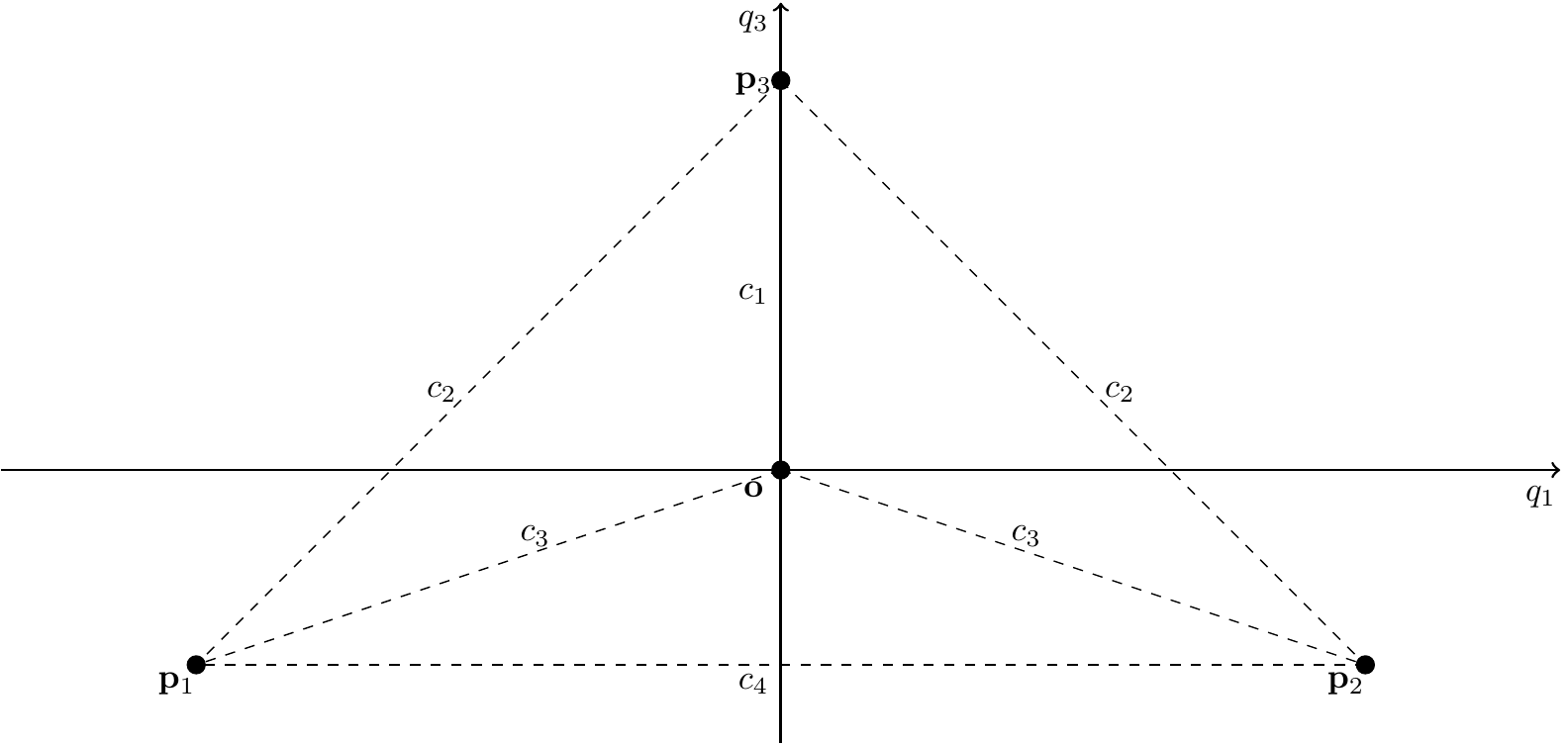}
	\caption{The four critical points of the
	potential~$F$, which is defined by Eq.~\ref{F}. The dashed lines indicate the
	``transition costs'' that are defined by Eq.~\ref{costs}}
	\label{fig:criticalpoints}
\end{figure}

\begin{proposition} \label{prop:1}
We have a critical point $(q_1^s, q_3^s)$, of the functional (\ref{eq:J}) in the admissible space (\ref{eq:admissible}), subject to the boundary conditions (\ref{eq:d2}) and (\ref{eq:d3}), such that $q_1 = 0$ on $x = 0$ and $y = 0$. This in turn defines a LdG critical point of the form (\ref{eq:d1}), referred to as a Well Order Reconstruction ``WORS'' critical point for a square with an isotropic inclusion.
\end{proposition}

\begin{proof}
 We follow the ideas in \cite{majcanevarispicer2017} and minimize  the functional $J[q_1, \, q_3 ]$ on a quadrant of the rotated rescaled square with an isotropic inclusion, as defined in (\ref{eq:domain}).
For the minimization problem on the quadrant, we need additional boundary conditions on the square diagonals. We impose the additional boundary condition that $q_1 = 0$ on the square diagonals $x = 0$ and $y = 0$. Further, we impose $\frac{\partial q_3}{\partial \mathbf{n}} = 0$ on $x = 0$ and $y = 0$, where $\mathbf{n}$ is the unit normal to the diagonals. We can prove the existence of a minimizer $\left(q_1^*, q_3^* \right)$ of $J[q_1, \, q_3 ]$ on the quadrant in $W^{1,2}$, subject to these boundary conditions, from the direct method in the calculus of variations \cite{evans}. We define $q_1^s$ on the square by an odd reflection of $q_1^*$ about the square diagonals and $q_3^s$ by an even reflection of $q_3^*$ about the square diagonals. By using the same arguments as in \cite{majcanevarispicer2017}, we can check that $\left(q_1^s, q_3^s \right)$ is a critical point of $J[q_1, \, q_3 ]$ on the square with an isotropic inclusion, with the property $q_1 =0$ on $x = 0$ and $y=0$. We label this as the ``Well Order Reconstruction Solution''.
\end{proof}

We define the WORS as being a LdG critical point given by
\begin{equation}
\label{eq:wors}
 \begin{split}
\Qvec_s = q_1^s(x, \, y) \left(\nvec_1 \otimes \nvec_1 - \nvec_2\otimes \nvec_2 \right) +
  q_3^s(x, \, y) \left(2 \hat{\mathbf{z}}\otimes\hat{\mathbf{z}} - \nvec_1\otimes \nvec_1 - \nvec_2\otimes \nvec_2 \right)
 \end{split}
\end{equation}
where the pair $(q_1^s, q_3^s)$ is defined above in Proposition~(\ref{prop:1}). 
There is an important distinction between the WORS for a square domain with and without an isotropic inclusion. In \cite{majcanevarispicer2017}, the authors study the WORS on a square domain without an isotropic inclusion and hence, only have the tangent uniaxial Dirichlet conditions (\ref{eq:bc1}) on the outer square edges in which case, we can have a WORS solution with constant $q_3$ at a special temperature defined by $A=-\frac{B^2}{3C}$. In this case, the WORS analysis reduces to a scalar variational problem as studied in \cite{majcanevarispicer2017}. In the case of a square with an isotropic inclusion, the boundary conditions for $q_3$ on the inner and outer square do not match and hence, we have inhomogeneous profiles for both $q_1$ and $q_3$ for all values of $A$, making this a harder problem. Next, we have a uniqueness result following the same arguments as in \cite{majcanevarispicer2017} and \cite{lamy2014}.

\begin{proposition} \label{prop:2}
The WORS defined in (\ref{eq:wors}) is the unique LdG critical point (and hence, globally stable) for either $\lambda$ sufficiently small or for $\rho$ sufficiently close to $1$ i.e. for either very small squares or for squares with inclusions with the aspect ratio approaching unity.
\end{proposition}
\begin{proof}
The proof follows by the arguments in Proposition~$4.2$ of \cite{lamy2014},
provided that we are able to bound the Poincar\'e constant of $\Omega$
in terms of the geometric parameter~$\rho$.
Let $u\in H^1(\Omega)$ be any scalar function such that~$u = 0$ on~$\partial\Omega$;
we extend $u$ out of~$\Omega$ by zero. We consider the set
$K : = \{(x, y)\in\Omega\colon x\geq 0, \, y\geq 0, \, \rho < x + y < 1\}$
and define new variables $(s, \, t)$ by 
\[
 x = ts, \qquad y = t(1-s)
\]
for each $(x, y)\in K$. The variables $(s, t)$ vary in the range $s\in (0, \, 1)$,
$t\in (\rho, \, 1)$.
We compute the integral of $|u|^2$ over~$K$ with respect to the coordinates
$(s, t)$, and apply the fundamental theorem of calculus in
$t$-direction, using that $u=0$ 
for $t = 1$:
\begin{align*}
  \int_K |u(\xvec)|^2 \, \d\xvec &= 
  \int_0^{1}\int_{\rho}^{1} t |u(s, t)|^2 \,\d t\,\d s
  \leq \int_0^{1}\int_{\rho}^{1} t \abs{\int_{t}^1 \partial_\xi
  u(s, \xi) \, \d\xi}^2 \,\d t\,\d s \\
  &\leq \int_0^{1}\int_{\rho}^{1}\int_{t}^{1} 
  t(1 - t) \abs{\partial_\xi u(s, \xi)}^2 \, \d\xi \,\d t\,\d s
\end{align*}
The last inequality follows by the H\"older inequality.
Now, we have $|\partial_t u|^2 = |s\partial_x u + (1-s)\partial_yu|^2
\leq s |\partial_x u|^2 + (1-s) |\partial_y u|^2 \leq |\nabla u|^2$,
where $\nabla$ denotes the gradient with respect to~$(x, y)$.
Using this with $\rho \leq t \leq \xi$, 
and reverting to the original coordinates $(x, \, y)$, we obtain
\begin{equation*} 
 \int_K |u(\xvec)|^2 \d\xvec \leq  
 (1-\rho)^2 \int_0^{1}\int_{\rho}^{1}
 \xi \abs{\partial_t u(s, \xi)}^2 \, \d\xi\,\d s
 \leq (1-\rho)^2 \int_K |\nabla u(\xvec)|^2 \d\xvec. 
\end{equation*}

By repeating the same argument on the other quadrants, and by adding the resulting inequalities, we conclude that
\begin{equation} \label{poincare}
 \int_{\Omega} |u(\xvec)|^2 \,\d\xvec \leq (1-\rho)^2 \int_{\Omega} |\nabla u(\xvec)|^2 \,\d\xvec. 
\end{equation}

By an application of the maximum principle, as in~\cite[Proposition~3]{amaz}, 
we know that any solution~$\Qvec$ of~\eqref{eq:6} in the admissible class~\eqref{eq:admissible} is bounded, i.e.
$|\Qvec(\xvec)|\leq M$ for any~$\xvec\in\Omega$ and a constant~$M$ 
that only depends on the coefficients~$A$, $B$, $C$. Now, by repeating
verbatim the arguments in \cite[Lemma 8.2]{lamy2014}, and using 
the Poincar\'e inequality~\eqref{poincare}, we conclude that the boundary
value problem~\eqref{eq:6}, \eqref{eq:bc0}, \eqref{eq:bc1},
\eqref{eq:bc2} has a unique solution, provided that
\[
 (1-\rho)^2 \lambda^2 < \kappa L
\]
for some positive constant $\kappa$ that only depends on~$M$, $A$, $B$, $C$.
\end{proof}

\section{The limit of large domains}
\label{sec:gamma}

In the following proposition, we analyze the asymptotic behavior
of minimizers of~\eqref{eq:J} in the limit as~$\lambda\to+\infty$.
To this end, we need to introduce some notation. 
We denote~$\qvec := (q_1, \, q_3)$ and define a metric~$d$ on the 
$\qvec$-plane in the following way: for any two points~$\qvec_0$,
$\qvec_1\in\Rr^2$, we let
\begin{equation} \label{dist}
 d(\qvec_0, \, \qvec_1) :=\inf\left\{\int_0^1 F^{1/2}(\qvec(t))
 |\qvec^\prime(t)|\,\d t\colon \qvec\in C^1([0, \, 1]; \, \Rr^2),
 \ \qvec(0) = \qvec_0, \ \qvec(1) = \qvec_1 \right\} \! .
\end{equation}
This is the geodesic distance associated with the Riemannian 
metric~$F^{1/2}$. However, this metric is degenerate, in that
$F^{1/2}(\pvec_1) = F^{1/2}(\pvec_2) = F^{1/2}(\pvec_3) = 0$ for
$\pvec_1$, $\pvec_2$, $\pvec_3$ given by~\eqref{critical_points}.
Despite the degeneracy, it can be proved that the infimum in~\eqref{dist} 
is actually achieved by a minimizing geodesic, for any~$\qvec_0$,
$\qvec_1\in\Rr^2$ (this follows by the arguments in 
\cite[Lemma~9]{Sternberg}). 

Let~$\mathcal{H}^1(E)$ denote the length of a set~$E\subseteq\Rr^2$
(or, more formally, its $1$-dimensional Hausdorff measure).
For every measurable subset~$E\subseteq\Omega$,
we denote by $\chi_E$ the characteristic function of~$E$ 
(i.e., $\chi_E(x) :=1$ for~$x\in E$, and~$\chi_E(x) := 0$ otherwise)
and by $\partial^* E$ the reduced boundary of~$E$, that is, the set 
of points~$x\in\partial E$ such that the limit
\[
 \nu_E(x) := \lim_{\rho\searrow 0}\frac{\D\chi_E(B_\rho(x))}{|\D\chi_E|(B_\rho(x))}
\]
exists and~$|\nu_E(x)|=1$. (Here $\D\chi_E$ stands for the 
distributional derivative of~$\chi_E$, which is a measure, and 
$|\D\chi_E|$ is the total variation measure; see, e.g., the book~\cite{evans}
for a detailed discussion on the distributional derivative.)
The reduced boundary is a subset $\partial^* E \subseteq\partial E$
with the following property:
\[
 \mathcal{H}^1(\partial^* E\cap\Omega) = \sup\left\{
 \int_E \mathrm{div} \, \varphi \,\d A \colon
 \varphi\in C^1_{\mathrm{c}}(\Omega), 
 \ |\varphi| \leq 1 \textrm{ on } \Omega\right\} \! .
\]
(see, e.g., \cite[Section~14]{Simon-GMT}).
If~$E$ has a regular (say, piecewise~$C^1$) boundary, 
by the Gauss-Green theorem the right-hand side of this formula reduces 
to~$\mathcal{H}^1(\partial E\cap\Omega)$, and 
indeed~$\partial^* E = \partial E$ in this case; however,
for a generic set~$E$ with non-regular boundary, we might 
have~$\partial^* E \subsetneq \partial E$.

Finally, we set~$\qvec_{\mathrm{b}} := (q_{1,\mathrm{b}}, \, q_{3,\mathrm{b}})$,
where~$q_{1,\mathrm{b}}$, $q_{3,\mathrm{b}}$ are defined by
\eqref{eq:d2}, \eqref{eq:d3} respectively. 
We let~$\qvec_\lambda := (q_1, \, q_3)$ be a minimizer of the 
functional~\eqref{eq:J}, for~$\lambda >0$.

\begin{proposition} \label{prop:Gamma}
 There exists a subsequence~$\lambda_j\nearrow +\infty$ such that
 $\qvec_{\lambda_j}$ converges, in~$L^1(\Omega)$ and a.e., to a map of
 the form
 \[
  \qvec_\infty = \sum_{k=1}^3 \pvec_k \, \chi_{E^*_k}.
 \]
 Here~$\pvec_1$, $\pvec_2$, $\pvec_3$ are defined 
 by~\eqref{critical_points}, and~$E^*_1$, $E^*_2$, $E^*_3$ are
 measurable, pairwise disjoint sets such that
 $\Omega = E^*_1\cup E^*_2 \cup E^*_3$. Moreover,
 $E^*_1$, $E^*_2$, $E^*_3$ minimize the following functional:
 \begin{equation} \label{opt_partition}
  J_\infty[E_1, \, E_2, \, E_3] := 
  \sum_{i, j=1}^3 d(\pvec_i, \, \pvec_j) \,
  \mathcal{H}^1(\partial^* E_i\cap \partial^* E_j\cap\Omega) 
  + \int_{\partial\Omega} d(\qvec_\infty(\rvec), \, \qvec_{\mathrm{b}}(\rvec))
  \, \d\mathcal{H}^1(\rvec)
 \end{equation}
 among all possible choices of measurable, pairwise disjoint sets $E_1$,
 $E_2$, $E_3$ such that $\Omega = E_1\cup E_2\cup E_3$.
\end{proposition}

The sets~$E^*_1$, $E^*_2$, $E^*_3$ give a partition of the domain $\Omega$,
and they are optimal, in the sense that they minimise 
the functional~\eqref{opt_partition}. This functional
depends on the length of the transition
layers~$\partial^* E_i\cap\partial^* E_j$ and on~$d(\pvec_i, \, \pvec_j)$, 
which represents the energy cost of a transition from the 
state~$\pvec_i$ to~$\pvec_j$. The functional~\eqref{eq:J} also 
contains a boundary term, which accounts for the possible presence
of boundary layers. 


\begin{proof}[Proof of Proposition~\ref{prop:Gamma}]
 This result can be shown using classical arguments in the theory of $\Gamma$-convergence.
 More precisely, Proposition~\ref{prop:Gamma} follows by the main result in~\cite{Baldo}
 (see also~\cite[Theorem~3.9]{FonsecaTartar} or~\cite[Theorem~7.20]{Braides}
 for similar results). The analysis in~\cite{Baldo, FonsecaTartar} does not take into account
 the presence of boundary conditions, such as~\eqref{eq:d2}--\eqref{eq:d3}. 
 However, these can be included by straightforward modifications of the arguments,
 as indicated in~\cite[Section~4.2.1 and Theorem~7.10]{Braides}.
\end{proof}

Let us introduce the transition costs
\begin{equation}\label{costs}
\begin{array}{l l}
c_1 := d(\mathbf{o}, \, \pvec_3), & \qquad c_2 := d(\mathbf{o}, \, \pvec_1) = d(\mathbf{o}, \, \pvec_2), \\
c_3 := d(\pvec_1, \, \pvec_3) = d(\pvec_2, \, \pvec_3), & \qquad c_4 := d(\pvec_1, \pvec_2), \\
\end{array}
\end{equation}
where $\mathbf{o} :=(0, \, 0)$ is the origin in the $(q_1, \, q_3)$-plane
and~$d$ is the intrinsic distance defined by~\eqref{dist}. 
These costs $c_1$, $c_2$, $c_3$, $c_4$ are functions of~$A$, $B$, $C$.
We have used the symmetry of the function~$F$, given by~\eqref{F}, to deduce that 
$d(\mathbf{o}, \, \pvec_1) = d(\mathbf{o}, \, \pvec_2)$ and 
$d(\pvec_3, \, \pvec_1) = d(\pvec_3, \, \pvec_2)$. 
By analyzing the possible configurations of~$(E_1, \, E_2, \, E_3)$, 
we can identify three candidate minimizers for ~\eqref{opt_partition}
and compute their energy as a function of the transition costs~\eqref{costs}.
For the sake of simplicity, in what follows we assume that~$\varepsilon =0$, i.e.
no truncation of the domain~$\Omega$ has been made. This is acceptable because,
when~$\varepsilon$ is small, the contribution of the truncated edges to
the boundary integral in~\eqref{opt_partition} is negligible.

\begin{itemize}
 \item A configuration with $\qvec = \pvec_1$ on~the first and third quadrant, and~$\qvec = \pvec_2$
 on the second and fourth quadrant (i.e., $E_1 = \{(x, \, y)\in\Omega\colon xy\geq 0\}$,
 $E_2 = \{(x, \, y)\in\Omega\colon xy< 0\}$, $E_3 =\emptyset$).
 This configuration corresponds to \emph{the $\lambda\to+\infty$ limit of the WORS}. 
 It has transition layers from the isotropic state to~$\pvec_1$ or~$\pvec_2$
 over the whole of the inner boundary~$\Gamma_1$, 
 and transition layers~$\pvec_1\to\pvec_2$ on the diagonals.
 Using (\ref{opt_partition}), the energy of this WORS-like configuration in the $\lambda\to+\infty$ limit is given by 
 \[
  J_\infty(\mathrm{WORS}) = 4\sqrt{2} \rho c_2 + 
     4\left(1 - \rho \right)c_4.
 \]
 
 \item Two configurations related by symmetry, with $\qvec=\pvec_1$,
 respectively $\qvec=\pvec_2$ over almost the entire domain ~$\Omega$. Equivalently, in terms of the~$E_k$'s,
 these configurations are given by $E_1=\Omega$, $E_2=E_3=\emptyset$
 and $E_2=\Omega$, $E_1=E_3=\emptyset$, respectively. These configuration have a transition layer 
 at the inner boundary, from the isotropic to a uniaxial state (either $\mathbf{p}_1$ or $\mathbf{p}_2$), and two 
 \emph{boundary transition layers} on the edges~$C_2$, $C_4$ respectively (or $C_1$, $C_3$),
 to account for the boundary conditions~\eqref{eq:d2}--\eqref{eq:d3}. We refer to these states as \emph{BD} states, to abbreviate for boundary distortion, since they have two distinctive edge transition layers along a pair of parallel outer square edges.
 These two BD configurations have the same energy given by
 \[
  J_\infty(\mathrm{BD}) = 4\sqrt{2}\, \rho c_2 + 2\sqrt{2}c_4.
 \]
 
 \item A configuration with~$\qvec=\pvec_3$ in a neighbourhood of the inner boundary,
 surrounded by the same cross structure as in the WORS, that is,
 \begin{align*}
   E_3 &= \left\{(x, \, y)\in\Omega\colon |x + y|\leq  \rho + \eta, \
      |x - y|\leq  \rho  + \eta \right\} \! , \\
   E_1 &= \left\{(x, \, y)\in\Omega\setminus E_3\colon xy\geq 0 \right\} \! , \\
   E_2 &= \left\{(x, \, y)\in\Omega\setminus E_3\colon xy< 0 \right\} \! .
 \end{align*}
 In this \emph{``escaped''} configuration, the isotropic core is surrounded by a uniaxial region,
 $E_3$, with positive order parameter or equivalently $q_3>0$. This may be energetically convenient, if
 the transition $\pvec_1\to\pvec_2$ is energetically very expensive,
 compared to the transitions $\mathbf{o}\to\pvec_3$ and 
 $\pvec_3\to\pvec_1$, $\pvec_3\to\pvec_2$. If this is the case, the escaped configuration
 reduces the length of the (very expensive) transition layer along the diagonals, 
 at the price of introducing a new transition layer near the core.
 The overall cost of this configuration is given by
 \[
  J_\infty(\mathrm{ESC}) = 4\sqrt{2} \rho c_1 +
     4\sqrt{2}\left( \rho + \eta\right) c_3 + 
     4\left(1 - \rho - \eta\right)c_4.
 \]
\end{itemize}

A configuration that has an island of the state $\mathbf{p}_3$ around the core,
surrounded by a constant state~$\pvec_1$ or~$\pvec_2$, 
 always has greater energy than the competing BD-configuration. This follows
from the triangle inequality for the metric~$d$, which gives~$c_1 + c_3 \geq c_2$. 
Therefore, we will not consider this configuration here. Other configurations, that have 
``two-steps transition layers'' e.g. a transition of the form $\pvec_1\to\pvec_3\to\pvec_2$
occurring along the diagonals, can be ruled out for the same reason.
Configurations with non-straight transition layers can also be ruled out,
as the energy per transition layer is proportional to the length of the transition layer
and a non-straight transition layer between two points has greater length than a straight layer.

Now, we can compare the energy costs of these configurations.
\begin{itemize}
 \item $J_\infty(\mathrm{WORS})<J_\infty(\mathrm{BD})$ if and only if $\rho > 1 - \sqrt{2}/2$.
 
 \item We compare $J_\infty(\mathrm{ESC})$ with $J_\infty(\mathrm{WORS})$. 
 By substituting the explicit expressions for the two energies, we see that the inequality
 $J_\infty(\mathrm{ESC})<J_\infty(\mathrm{WORS})$ is equivalent to
 \[
   (\sqrt{2}c_3 - c_4)\eta < \sqrt{2}(c_2 - c_1 - c_3) \rho
 \]
 and the right-hand side is always non-positive, due to the triangle inequality.
 Thus, for this inequality to be satisfied we must have $c_4 > \sqrt{2}c_3$.
 By imposing the geometric constraint that $\eta \leq 1 - \rho$, 
 we obtain 
 \[
  \frac{\sqrt{2}(c_1 - c_2 + c_3)}{(c_4 - \sqrt{2}c_3)} \rho
  \leq 1 - \rho.
 \]
 By straightforward algebraic manipulations, we conclude that
 the inequality $J_\infty(\mathrm{ESC})<J_\infty(\mathrm{WORS})$ 
 holds if and only if
 \begin{equation*}
  \begin{cases}
   c_4 > \sqrt{2} c_3 \\
   0 < \rho < R_1
  \end{cases}
  \quad \textrm{or} \quad
  \begin{cases}
   c_4 > \sqrt{2} c_3 \\
   R_1 < 0,
  \end{cases}
 \end{equation*}
 where
 \begin{equation*}
  R_1 := \frac{c_4 - \sqrt{2}c_3}{\sqrt{2}c_1 - \sqrt{2}c_2 + c_4}.
 \end{equation*}
 
 \item Arguing in a similar way, we conclude that
 $J_\infty(\mathrm{ESC}) < J_\infty(\mathrm{BD})$ if and only if
 \begin{equation*}
  \begin{cases}
   c_4 > \sqrt{2} c_3 \\
   c_2 > c_1 \\
   R_2 < \rho < 1 \\
  \end{cases}
  \qquad \textrm{or} \qquad
  \begin{cases}
   c_4 > \sqrt{2} c_3 \\
   c_1 > c_2 \\
   0 < \rho < R_2, \\
  \end{cases}
 \end{equation*}
 where
 \begin{equation*}
  R_2 := \frac{c_4 - 2 c_3}{2 c_1 - 2 c_2}.
 \end{equation*}
\end{itemize}

\section{Numerics}
\label{sec:numerics}
Let $\bar{\lambda}^2 = \dfrac{2C \lambda^2}{L}$, and we take
\begin{equation}
B = 0.64 \times 10^4 \mathrm{Nm}^{-2}, \quad C = 0.35 \times 10^4 \mathrm{Nm}^{-2}, \quad A = - \frac{B^2}{3C}, \quad \bar{\lambda}^2 = 200
\end{equation}
throughout this section if not stated differently. We choose this special value of $A$ because in the absence of a square inclusion, the WORS has a particularly simple parametrization in terms of a single variable $q_1$ and constant $q_3$ (see (\ref{eq:d1})) at this temperature \cite{majcanevarispicer2017}. In fact, this is the only temperature for which the system (\ref{eq:d4}) has a solution with constant $q_3$. Of course, we cannot have solutions with constant $q_3$ for this model problem because of the inhomogeneous boundary conditions but we still regard this temperature as a special reference point which allows for easy comparison with the results in \cite{majcanevarispicer2017}. We assume that $\bar{\lambda}^2 = 200$ is large enough for the asymptotic estimates in Section \ref{sec:gamma} to hold; we have also checked the trends with larger values of $\bar{\lambda}^2$ and they are qualitatively the same.

\subsection{Transition Costs} 

First, we compute the transition costs defined in (\ref{costs}). 
According to standard arguments in Riemannian geometry,
the intrinsic distance $d(\qvec_0, \, \qvec_1)$ defined in (\ref{dist}) can be calculated alternatively as
\begin{equation*} \label{dist_1}
 d(\qvec_0, \, \qvec_1) =\inf\left\{ \left( \int_0^1 F(\qvec(t))|\qvec^\prime(t)|^2 \,\d t \right)^{1/2} ~\bigg|~ \qvec\in C^1([0, \, 1]; \, \Rr^2), \ \qvec(0) = \qvec_0, \ \qvec(1) = \qvec_1 \right\} \, .
\end{equation*}
The profiles of geodesic $\qvec(t) = (q_1(t), q_3(t))$ in each case of (\ref{costs}) are shown in Fig. \ref{F1}; these are the optimal profiles which minimise the intrinsic distance between the four critical points $\mathbf{o}, \mathbf{p}_1, \mathbf{p}_2, \mathbf{p}_3$ and the associated costs are given below:
\begin{equation}
c_1 = 22.3067, \quad c_2 = 34.7378, \quad c_3 = 41.6817, \quad c_4 = 60.2955.
\end{equation}
\begin{figure}[!h]
\centering
\includegraphics[width = 0.7 \linewidth]{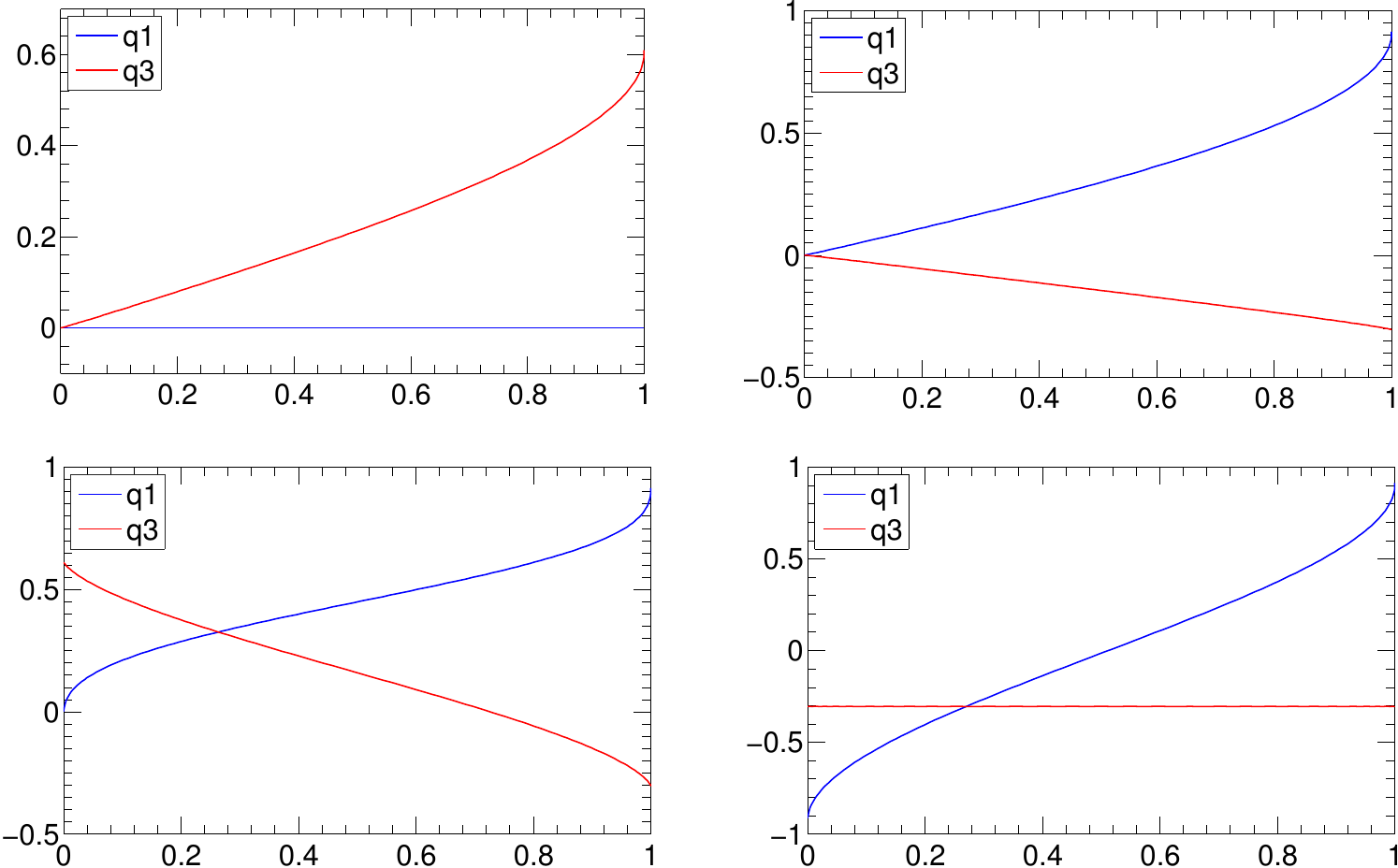}
\caption{The profiles of geodesic $\qvec(t) = (q_1(t), q_3(t))$ for different $\qvec_0$ and $\qvec_1$ ($A = - \frac{B^2}{3C}$).}\label{F1}
\end{figure} 

Hence,
\begin{equation}
  c_2 > c_1, \quad c_4 > \sqrt{2} c_3, \quad R_2 > R_1.
\end{equation}
 In view of the discussion in the previous section, $$\min\{J_{\infty}({\rm WORS}), \, J_{\infty}({\rm BD})\} < J_{\infty}({\rm ESC})$$ requires that
\begin{equation}\label{ESC_no_min_1}
R_2 < \rho < R_1, ~~\text{if}~ c_4 > \sqrt{2} c_3 ~\text{and}~ c_2 > c_1,
\end{equation}
which cannot hold since $R_2 > R_1$. Therefore, ESC cannot be energetically preferred to either the WORS or BD for this choice of parameters.

Next, we perform a systematic search of the parameter space in terms of $A$, for fixed $B$ and $C$; the transition costs $c_i$ and the 
quantities $R_i$, as a function of the reduced temperature $t = \frac{27 AC}{B^2}$, are numerically computed and plotted in Fig. \ref{F2}. We note that $c_2 > c_1$ and $R_2 > R_1$ hold true in all the numerical simulations. If $c_4 > \sqrt{2}c_3$, then the same arguments as above apply to exclude the ESC as a competitor for an energy minimizer; if $c_4 <\sqrt{2}c_3$, then the ESC cannot be energetically preferred to the WORS or BD according to the estimates in the previous section.  



\begin{figure}[!htb]
 \centering
 \includegraphics[width = 0.75 \linewidth]{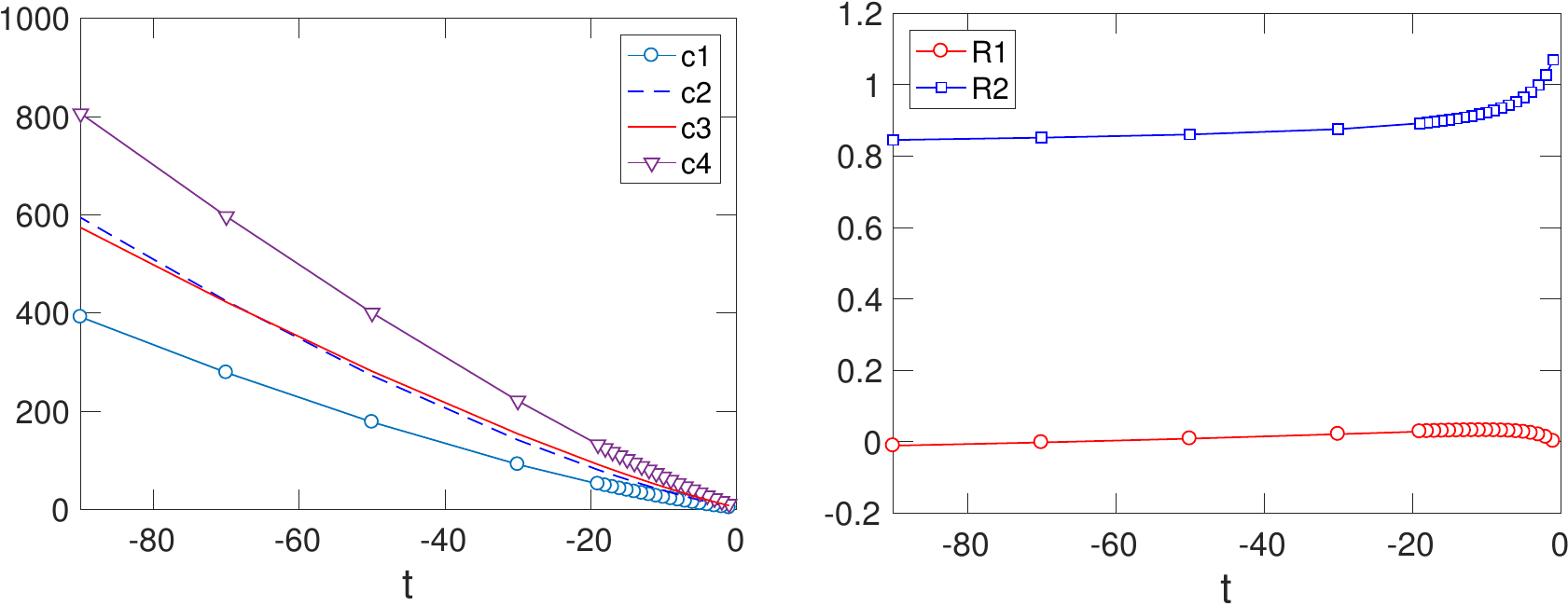}
\caption{The value of $c_i$ and $R_i$ as a function of $t = \frac{27 AC}{B^2}$ }\label{F2}
\end{figure}

\subsection{WORS and BD on a square with an isotropic core}
For the following simulations, we take
\begin{equation}
\Omega = \{(x, y) \in \mathbb{R}^2\colon \rho < \max\{|x|, \, |y|\} < 1 \},
\end{equation}
as the computational domain and seek numerical solutions of the form
\begin{equation}\label{q13}
\Qvec(x,y) = q_1(x, y)(\e_x \otimes \e_x - \e_y \otimes \e_y)  + q_3(x, y) (2 \e_z \otimes \e_z - \e_x \otimes \e_x - \e_y \otimes \e_y), 
\end{equation}
where $\e_x$, $\e_y$ and $\e_z$ are unit-vectors in the $x$-, $y$- and $z$-directions respectively, subject to the boundary conditions
\begin{equation}\label{BC-general}
\begin{aligned}
& \Qvec(x, y) = 0 \qquad \quad \text{on}~\Gamma_{\rm in}, \\
& \Qvec(x, \pm 1) = ~~\frac{s_{+}}{2} (\e_x \otimes \e_x - \e_y \otimes \e_y) - \frac{s_{+}}{6} (2 \e_z \otimes \e_z - \e_x \otimes \e_x - \e_y \otimes \e_y), \\
& \Qvec(\pm 1, y) = - \frac{s_{+}}{2} (\e_x \otimes \e_x - \e_y \otimes \e_y) - \frac{s_{+}}{6} (2 \e_z \otimes \e_z - \e_x \otimes \e_x - \e_y \otimes \e_y). \\
\end{aligned}
\end{equation}

For LdG critical points of the form (\ref{q13}), the Euler-Lagrange equations (\ref{eq:6}) reduce to 
\begin{equation}\label{eq_q13}
\begin{cases}
\Delta q_1 = \bar{\lambda}^2 \left( \dfrac{A}{2C} q_1 + \dfrac{B}{C} q_1 q_3 + (q_1^2 + 3 q_3^2) q_1 \right) \\  
\Delta q_3 = \bar{\lambda}^2 \left( \dfrac{A}{2C} q_3 + \dfrac{B}{C} (\frac{1}{3}q_1^2 - q_3^2) + (q_1^2 + 3 q_3^2) q_3 \right). \\
\end{cases}
\end{equation} 

We use a standard finite-difference method and Newton's Method to solve the system of coupled partial differential equations (\ref{eq_q13}). We plot the profiles for $q_1$, $q_3$ and 
biaxiality parameter $$\beta^2 = 1 - 6 \frac{\left( \tr(\Qvec^3) \right)^2}{\left( \tr(\Qvec^2) \right)^3}$$ in the WORS and BD for $\bar{\lambda}^2 = 200$ and $\rho = 0.2$ in Fig. \ref{BD_OR}. 
The biaxiality parameter $\beta^2 \in \left[0, 1 \right]$ for $\Qvec \neq \mathbf{0}$, $\beta^2$ vanishes when $\Qvec$ has two degenerate non-zero eigenvalues, and $\beta^2$ is unity when one of the eigenvalues vanishes and the corresponding $\Qvec$ is maximally biaxial \cite{majcanevarispicer2017}.
\begin{figure}[!hbt]
\centering
\includegraphics[width = 0.8\textwidth]{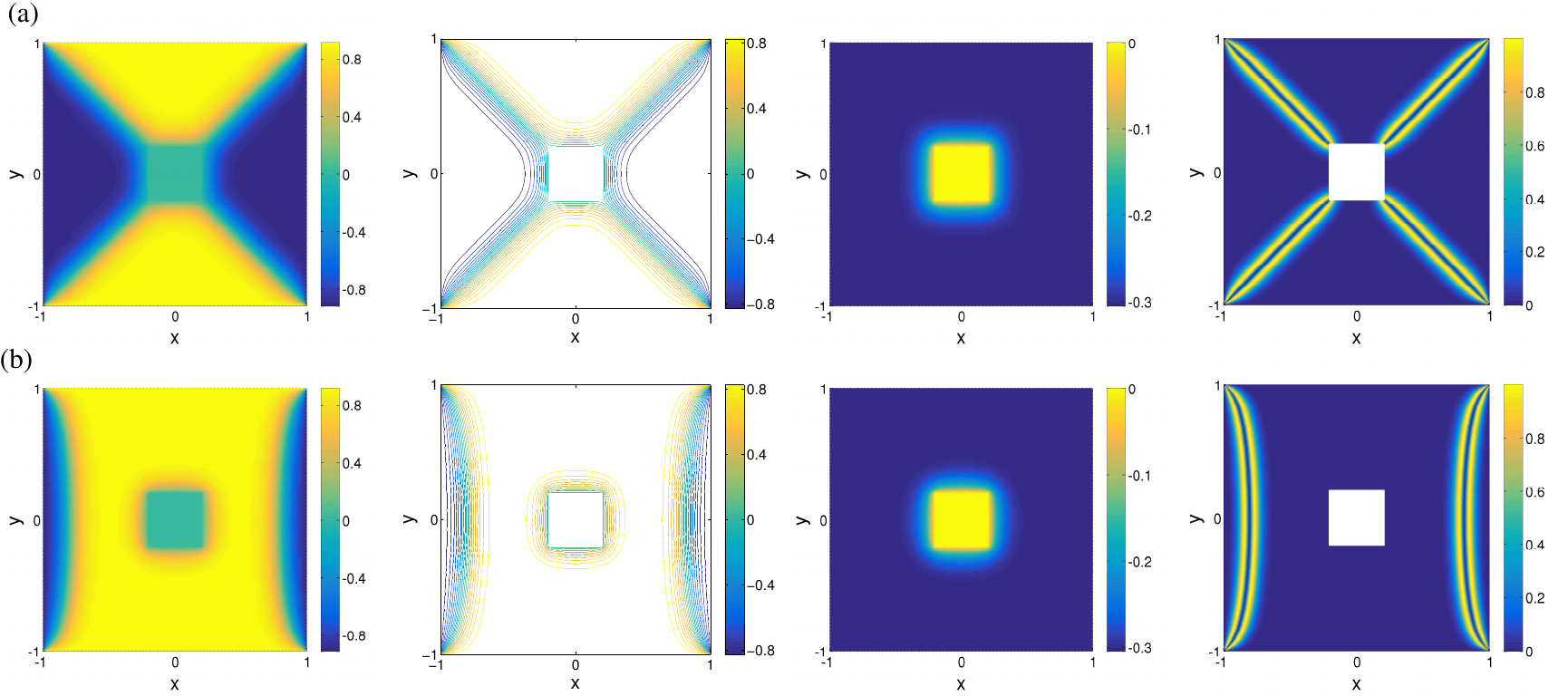}
\caption{(a) WORS for $\bar{\lambda}^2 = 200$ and $\rho = 0.2$. Left to right: plot and contour plot of $q_1$, plot of $q_3$, and plot of biaxiality parameter $\beta^2$. 
(b) BD for $\bar{\lambda}^2 = 200$ and $\rho = 0.2$. Left to right: plot and contour plot of $q_1$, plot of $q_3$, and plot of biaxiality parameter $\beta^2$.}\label{BD_OR}
\end{figure}
The WORS has a uniaxial cross with negative order parameter, connecting the vertices of the inner square and the outer square. The BD solution is distinguished by a pair of edge transition layers, localized near $x = \pm 1$ (or $y = \pm 1$). In both cases, $q_3$ decreases monotonically from zero on the inner boundary to $q_3 = -\frac{B}{6C}$ on the outer boundary.

We compare the free energies of BD and WORS for $\bar{\lambda}^2 = 200$ and various $\rho$ in Fig. \ref{BD_OR-energy}(a), which shows that WORS is energetically preferred for relatively large $\rho$. Indeed, the $\Gamma$-convergence argument in the previous section shows that, in the limit $\bar{\lambda}^2 \rightarrow \infty$, we have $J_{\infty}({\rm WORS}) < J_{\infty}({\rm BD})$ if and only if  $\rho > 1 - \sqrt{2}/2$. Numerically, we compute the critical value $\rho_0(\bar{\lambda}^2)$, such that $J_{\bar{\lambda}^2}({\rm BD}) = J_{\bar{\lambda}^2}({\rm WORS})$ when  $\rho = \rho_0(\bar{\lambda}^2)$, as a function of $\bar{\lambda}^2$ in Fig. \ref{BD_OR-energy}(b).
Qualitatively, we see that $\rho_0(\bar{\lambda}^2) \rightarrow 1 - \sqrt{2}/2$ when $\bar{\lambda}^2 \rightarrow \infty$, in agreement with the $\Gamma$-convergence results in the previous section. 

Since WORS is the unique LdG critical point for  either $\lambda$ sufficiently small or for $\rho$ sufficiently close to $1$, BD cannot be a critical point of the functional \eqref{eq:J} for either large $\rho$ or small $\bar{\lambda}^2$. Numerically, we find that for each $\bar{\lambda}^2$, there exists a critical value $\rho_1(\bar{\lambda}^2)$, for which BD is no longer a critical point of the functional \eqref{eq:J} when $\rho \geq \rho_1(\bar{\lambda}^2)$. This critical value $\rho_1(\bar{\lambda}^2)$ is found by increasing $\rho$ gradually till we cannot numerically obtain a  BD solution with a BD-like initial guess, even with the deflation technique \cite{farrell2015deflation}.
For $\bar{\lambda}^2 = 100$, $\rho_1 \approx 0.28$, whilst for the $\bar{\lambda}^2 = 200$, $\rho_1 \approx 0.42$. $\rho_1(\bar{\lambda}^2)$ as a function of $\bar{\lambda}^2$ is shown in Fig. \ref{BD_OR-energy}(c). By adapting the arguments in \cite{lamy2014} to a truncated square annulus such as ours (see also the proof of Proposition~\ref{prop:2} for more details), we can show that 
the LdG energy (\ref{eq:2}) is strictly convex for
$$ 1 - C_1\bar{\lambda}^{-1} < \rho < 1, $$
where $C_1$ is a positive constant independent of $\rho$ and $\bar{\lambda}^2$. 
Therefore, the LdG energy has a unique critical point, which is the WORS, for $\rho$ in this range and as $\bar{\lambda}^2$ increases, this range becomes narrower as illustrated by the estimate above.


\begin{figure}[!htb]
\centering
\includegraphics[width = \linewidth]{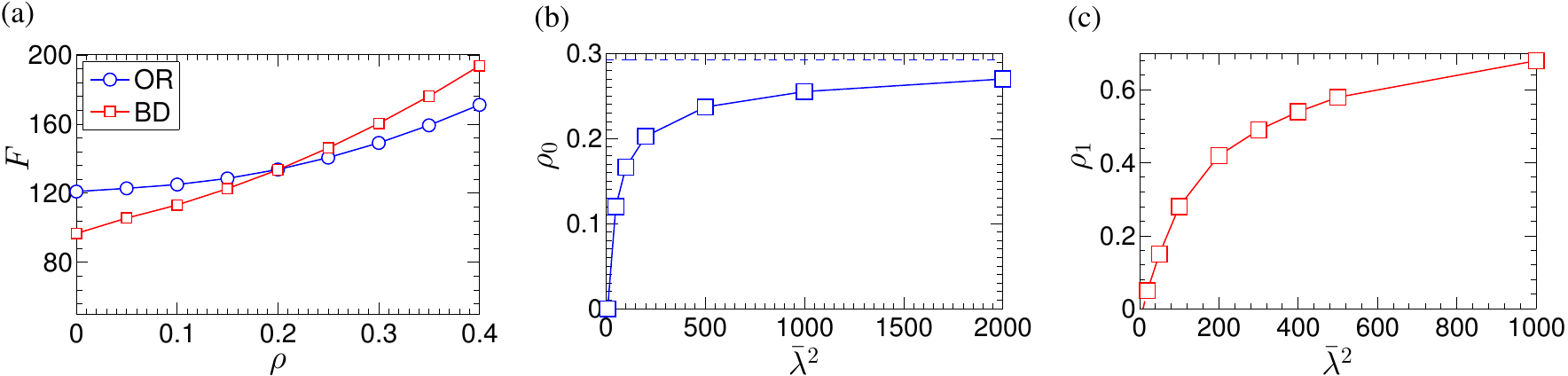}
\caption{(a) the free energy of {\rm WORS} and {\rm BD} for various $\rho$ for $\bar{\lambda}^2 = 200$. (b) the  critical value $\rho_0$, for which $J_{\bar{\lambda}^2}({\rm BD}) = J_{\bar{\lambda}^2}({\rm WORS})$ when $\rho > \rho_0$. (c)the  critical value $\rho_1(\bar{\lambda}^2)$, for which BD is no longer a critical point of the functional \eqref{eq:J} when $\rho \geq \rho_1(\bar{\lambda}^2)$.}\label{BD_OR-energy}
\end{figure}

We test the stabilities of the WORS and BD by solving the gradient flow equations for $q_1$ and $q_3$ in $\Omega$ as shown below:
\begin{equation}\label{grad_q13}
\begin{cases}
\pp_t q_1 = \Delta q_1 - \bar{\lambda}^2 \left( \dfrac{A}{2C} q_1 + \dfrac{B}{C} q_1 q_3 + (q_1^2 + 3 q_3^2) q_1 \right) \\  
\pp_t q_3 = \Delta q_3 - \bar{\lambda}^2 \left( \dfrac{A}{2C} q_3 + \dfrac{B}{C} (\dfrac{1}{3}q_1^2 - q_3^2) + (q_1^2 + 3 q_3^2) q_3 \right), \\
\end{cases}
\end{equation}
for $\bar{\lambda}^2 = 200$, subject to the Dirichlet boundary conditions~\eqref{eq:bc0}, \eqref{eq:bc1} and \eqref{eq:bc2} and different initial conditions. We use a standard finite-difference method for the spatial derivatives and the Crank-Nicolson scheme~\cite{iserles2009first} for time-stepping in the numerical simulations.

In Fig. \ref{OR}, we solve (\ref{grad_q13}) with a WORS-like initial condition as described below
\begin{equation}\label{eq:wors_ic}
q_1(x, y) = 
\begin{cases}
~~ s_{+}/2 & \textrm{for } -|y| < x < |y| \\
- s_{+}/2 & \textrm{for } -|x| < y < |x|, \\
\end{cases}
\quad q_3(x, y) = -s_{+}/6, \quad \forall (x, y) \in \Omega
\end{equation}
for $\bar{\lambda}^2 = 200$ with $\rho = 0.02$ and $\rho = 0.1$, respectively. The dynamic evolutions of $q_1$ in both cases are shown in Fig. \ref{OR}. For both cases, $\rho$ is in the range for which  $J({\rm WORS}) > J({\rm BD})$ according to Fig. \ref{BD_OR-energy}(a) and yet the dynamic evolutions are different for $\rho = 0.02$ and $\rho = 0.1$. For $\rho = 0.02$, the initial condition with the diagonal cross (see \eqref{eq:wors_ic}), evolves to BD, which indicates that WORS is unstable when $\rho$ is very small. However, for $\rho = 0.1$, the solution converges to the WORS although WORS has higher free energy than BD, which indicates that the WORS is metastable with a basin of attraction.
\begin{figure}[!htb]
\centering
\includegraphics[width = 0.8\linewidth]{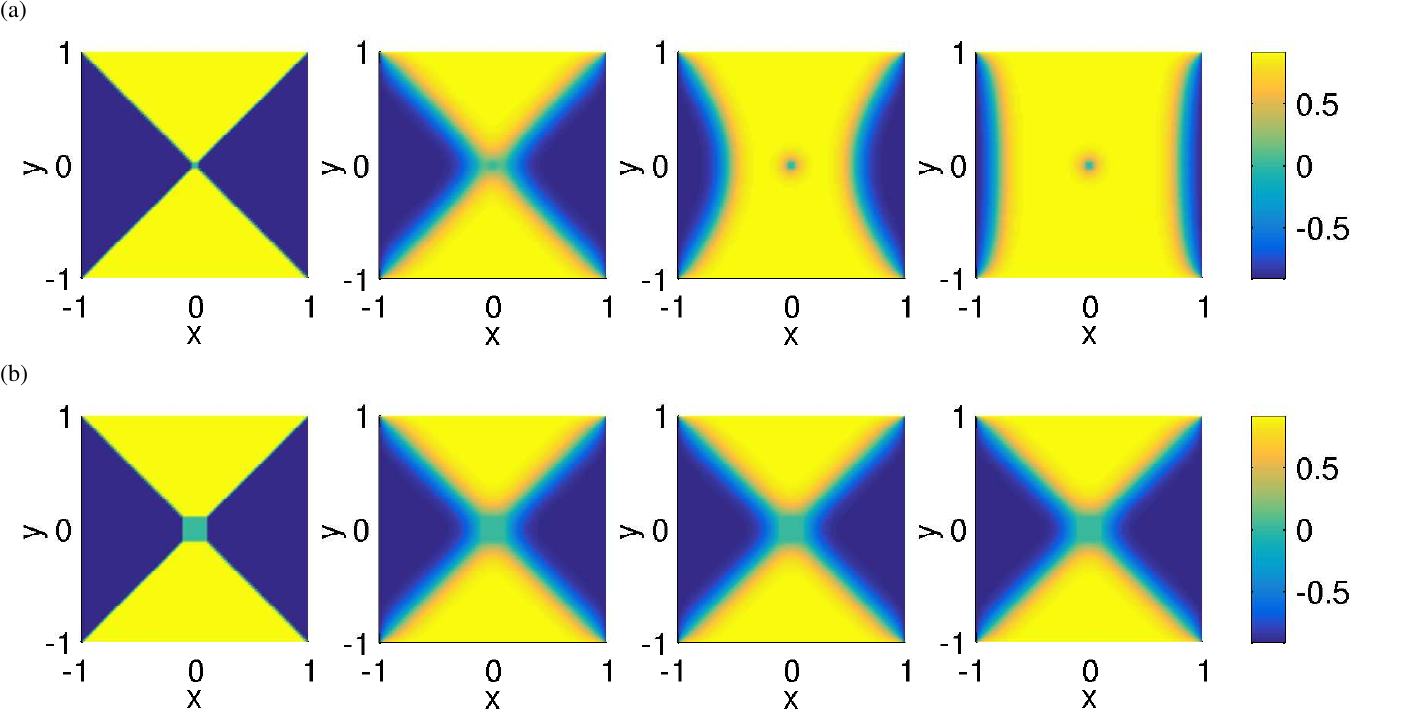}
\caption{(a) The profiles of $q_1$ for t = 0, t = 0.2, t = 0.5 and t = 2($\rho = 0.02$, $\bar{\lambda}^2 = 200$). 
(b) The profiles of $q_1$ for t = 0, t = 0.2, t = 0.5 and t = 2($\rho = 0.1$, $\bar{\lambda}^2 = 200$). 
}\label{OR}
\end{figure}

We also solve  (\ref{grad_q13}) with a BD-like initial condition for $\bar{\lambda}^2 = 200$ with $\rho = 0.4$ and $\rho = 0.44$. The dynamic evolutions of $q_1$ are displayed in Fig. \ref{BD}, for both cases. The previous discussions illustrate that BD ceases to be a critical point of the functional \eqref{eq:J} for $\rho \gtrsim 0.42$. We choose two values of $\rho$ that are at either end of this critical value. For $\rho = 0.4$, the numerical solution converges to BD, although BD has higher free energy than WORS, which indicates that the BD state is metastable. For $\rho = 0.44$, for which there is no BD-type critical point, the solution converges to WORS as expected.
\begin{figure}[!htb]
\centering
\begin{overpic}[width = 0.8 \linewidth]{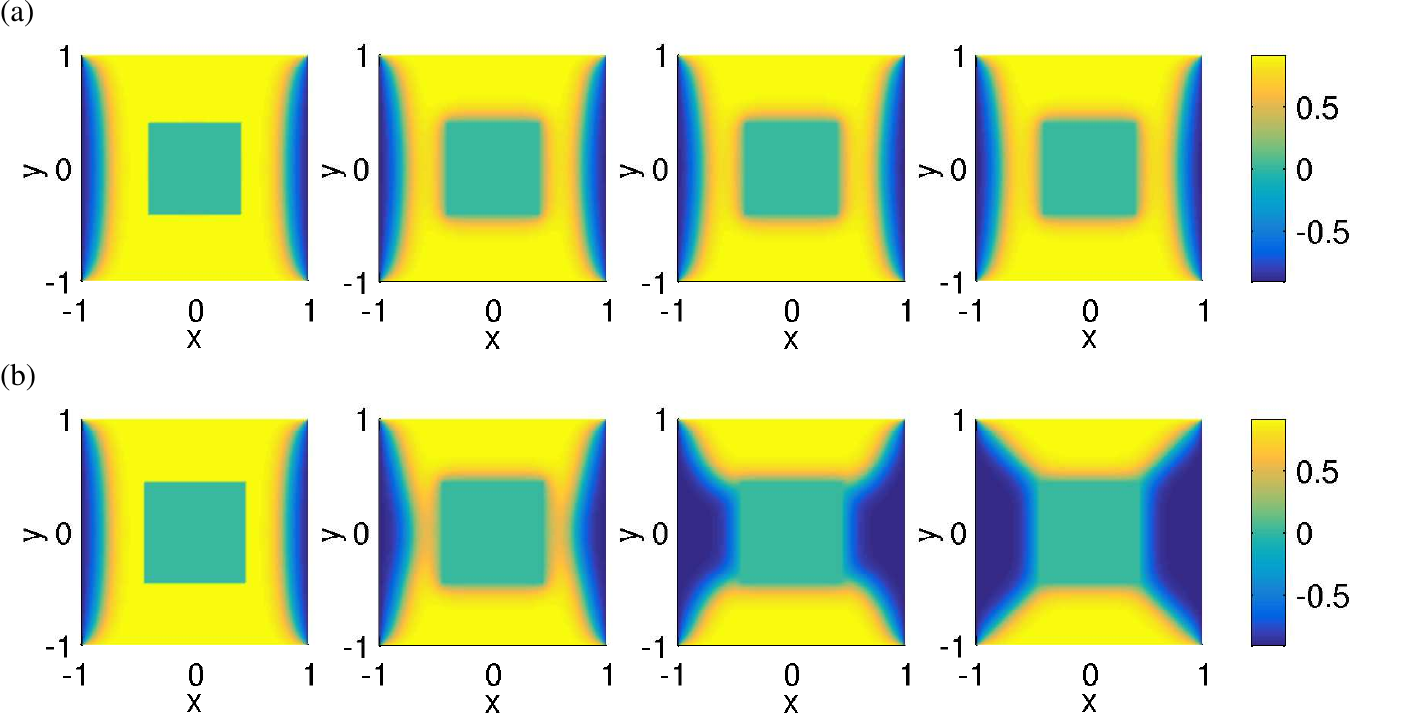}
\end{overpic}
\caption{(a) The profiles of $q_1$ for t = 0, t = 0.3, t = 0.4 and t = 2
($\rho = 0.4$, $\bar{\lambda}^2 = 200$). 
(b) The profiles of $q_1$ for t = 0, t = 0.3, t = 0.4 and t = 2 
($\rho = 0.44$, $\bar{\lambda}^2 = 200$). 
}\label{BD}
\end{figure}

\subsection{Decomposition of the Second Variation of the LdG energy}
The gradient flow simulations give us some information about the stabilities 
of WORS and BD in the restricted class of $\Qvec$ that have the form (\ref{q13}). In the following, we consider the second variation of the LdG energy (\ref{eq:rescaled}) 
about the WORS and BD-solutions, for arbitrary perturbations with five degrees of freedom. As is standard in variational problems in the calculus of variations, a solution is locally stable if the second variation of the LdG energy is positive for all admissible perturbations and a solution is unstable if we can find a perturbation for which the second variation is negative.
Consider a perturbation about the WORS or BD solutions of the form $\mathbf{W} = \Qvec + \epsilon \mathbf{V}$, where $\mathbf{V}$ vanishes at the boundary. The second variation of the LdG energy is given by:
\begin{equation}
\delta^2 F(\mathbf{V}) = \int_{\Omega} \dfrac{\lambda^2}{L} \left( A |\mathbf{V}|^2 - 2 B Q_{ij}V_{jk}V_{ki} + C |\Qvec|^2|\mathbf{V}|^2 + 2 C (\Qvec \cdot \mathbf{V})^2 \right) + |\nabla \mathbf{V}|^2  \dd \x.
\end{equation}
We write $\mathbf{V}$ as (see \cite{majcanevarispicer2017})
\begin{equation}
\begin{aligned}
\mathbf{V}(x,y) & = v_1(x, y)(\e_x \otimes \e_x - \e_y \otimes \e_y) + v_2(x, y) (\e_x \otimes \e_y + \e_y \otimes \e_x) \\
       & + v_3(x, y) (2 \e_z \otimes \e_z - \e_x \otimes \e_x - \e_y \otimes \e_y) \\
       & + v_4(x, y)(\e_x \otimes \e_z + \e_z \otimes \e_x) + v_5(x, y)(\e_y \otimes \e_z + \e_z \otimes \e_y ), \\
\end{aligned}
\end{equation}
where we treat the functions, $v_1 \ldots v_5$, as perturbations in the five independent basis directions.


For LdG critical points with $q_2 = q_4 = q_5 = 0$, such as the WORS and BD solutions with a constant eigenframe, we have
\begin{equation}\label{Per_V}
  \begin{aligned}
    \delta^2 F(\mathbf{V}) = \int_{\Omega} & \bar{\lambda}^2 \Biggl( \frac{A}{C} (v_1^2 + v_2^2 + 3 v_3^2 + v_4^2 + v_5^2) \\
    & -  \frac{B}{C}  \bigl( q_1 (v_4^2 - v_5^2) - 2 q_3 (v_1^2 + v_2^2) + 6q_3v_3^2 + q_3(v_4^2 + v_5^2) - 4q_1v_1v_3 \bigr) \\
    & + 2 \left( q_1^2 + 3 q_3^2 \right)(v_1^2 + v_2^2 + 3 v_3^2 + v_4^2 + v_5^2) + 4 (q_1v_1 + 3 q_3 v_3)^2 \Biggr) \\
    & + \left( 2 |\nabla v_1|^2 + 2 |\nabla v_2|^2 + 6 |\nabla v_3|^2 + 2 |\nabla v_4|^2 + 2 |\nabla v_5|^2 \right) \dd \x. \\
    \end{aligned}    
\end{equation}
where $v_i \in W_0^{1,2} \left(\Omega\right)$.

We can write (\ref{Per_V}) as
\begin{equation}
\delta^2 F(\mathbf{V}) = \delta^2 F(v_1, v_3) + \delta^2 F (v_2) + \delta^2 F(v_4) + \delta^2 F(v_5),
\end{equation}
where
\begin{equation}
\begin{aligned}
& \delta^2 F(v_1, v_3) = \int_{\Omega} \bar{\lambda}^2 \Biggl( \left( \frac{A}{C} + \frac{2B}{C} q_3 + 6 (q_1^2 + q_3^2) \right) v_1^2  + \left( \frac{3A}{C} - \frac{6B}{C}q_3 + 6 q_1^2 + 54 q_3^2 \right) v_3^2  \\
& \qquad \qquad \qquad \qquad \qquad  + \left( \frac{4B}{C} q_1   + 24 q_1 q_3 \right)  v_1v_3 \Biggr) + \left( 2 |\nabla v_1|^2 + 6 |\nabla v_3|^2 \right) \dd \x, \\
& \delta^2 F(v_2)  = \int_{\Omega} \bar{\lambda}^2 \Biggl( \frac{A}{C}  + \frac{2 B}{C} q_3 + 2 \left( q_1^2 + 3 q_3^2 \right) \Biggr) v_2^2 + 2 |\nabla v_2|^2  \dd \x, \\
& \delta^2 F(v_4)  = \int_{\Omega} \bar{\lambda}^2 \Biggl( \frac{A}{C} - \frac{B}{C} (q_1 + q_3) + 2 \left( q_1^2 + 3 q_3^2 \right) \Biggr) v_4^2  + 2 |\nabla v_4|^2 \dd \x, \\
& \delta^2 F(v_5)  = \int_{\Omega}  \bar{\lambda}^2 \Biggl( \frac{A}{C} - \frac{B}{C} (q_3 - q_1) + 2 \left( q_1^2 + 3 q_3^2 \right) \Biggr) v_5^2 + 2 |\nabla v_5|^2 \dd \x. \\
\end{aligned}
\end{equation}

Define
\begin{equation*}
\begin{aligned}
& \mathcal{V}_{13} =  \left \{ v_1 \left( \e_x \otimes \e_x - \e_y \otimes \e_y \right) + v_3 \left( 2 \e_z \otimes \e_z - \e_x \otimes \e_x - \e_y \otimes \e_y \right) \right \}, \\
& \mathcal{V}_2 = \left \{v_2 \left( \e_x \otimes \e_y + \e_y \otimes \e_x \right) \right \},~~\mathcal{V}_4 =  \left \{v_4 \left( \e_x \otimes \e_z + \e_z \otimes \e_x \right) \right \},~~\mathcal{V}_5 =  \left \{v_5 \left( \e_y \otimes \e_z + \e_z \otimes \e_y \right) \right \}, \\
\end{aligned}
\end{equation*}
which are  subspaces of $S_0$. We can consider perturbations in each subspace respectively. The perturbations in $\mathcal{V}_{13}$ do not distort the constant eigenframes of the WORS or BD solutions, the perturbations in $\mathcal{V}_{2}$ are in-plane perturbations of the eigenframe and the perturbations in $\mathcal{V}_4$ and $\mathcal{V}_5$ are out-of-plane perturbations of the eigenframe.


Firstly, we consider $\delta^2 F(v_1, v_3)$, which can be regarded as a functional of $v_1$ and $v_3$, for given $q_1$ and $q_3$. We can minimize $\delta^2 F(v_1, v_3)$ by solving the gradient flow equations
\begin{equation}\label{grad_v13}
\begin{cases}
\dfrac{\pp v_1}{\pp t} = \Delta v_1 - \bar{\lambda}^2 \left( \dfrac{1}{2}C_{11}(x, y) v_1 + \dfrac{1}{4}C_{13}(x, y) v_3 \right)   \\
\dfrac{\pp v_3}{\pp t} = \Delta v_3 - \bar{\lambda}^2 \left( \dfrac{1}{24} C_{13}(x, y) v_1 + \dfrac{1}{12} C_{33}(x, y) v_3 \right), \\
\end{cases}
\end{equation}
where
\begin{equation}
\begin{aligned}
& C_{11}(x, y) = \frac{A}{C} + \frac{2B}{C} q_3 + 6 q_1^2 + 6 q_3^2, \quad C_{13}(x, y) = \frac{4 B}{C} q_1 + 24 q_1 q_3, \\
& C_{33}(x, y) = \frac{3A}{C} - \frac{6B}{C}q_3 + 6 q_1^2 + 54q_3^2.  \\
\end{aligned}
\end{equation}
For $\bar{\lambda}^2 = 200$, WORS is a critical point for
$0 \leq \rho < 1$, but is unstable for small-$\rho$. 
In Fig.~\ref{C123_v13}(a), we plot $C_{11}(x, y), C_{13}(x, y)$ and $C_{33}(x, y)$ for the WORS solution, using the numerically computed $q_1$ and $q_3$ corresponding to the WORS with $\rho = 0.02$. It is relatively straightforward to find $v_1$ and $v_3$ such that
$\delta F(v_1, v_3) < 0$. An example is shown in Fig. \ref{C123_v13}(b).
\begin{figure}[!h]
\centering
\begin{overpic}[width = \linewidth]{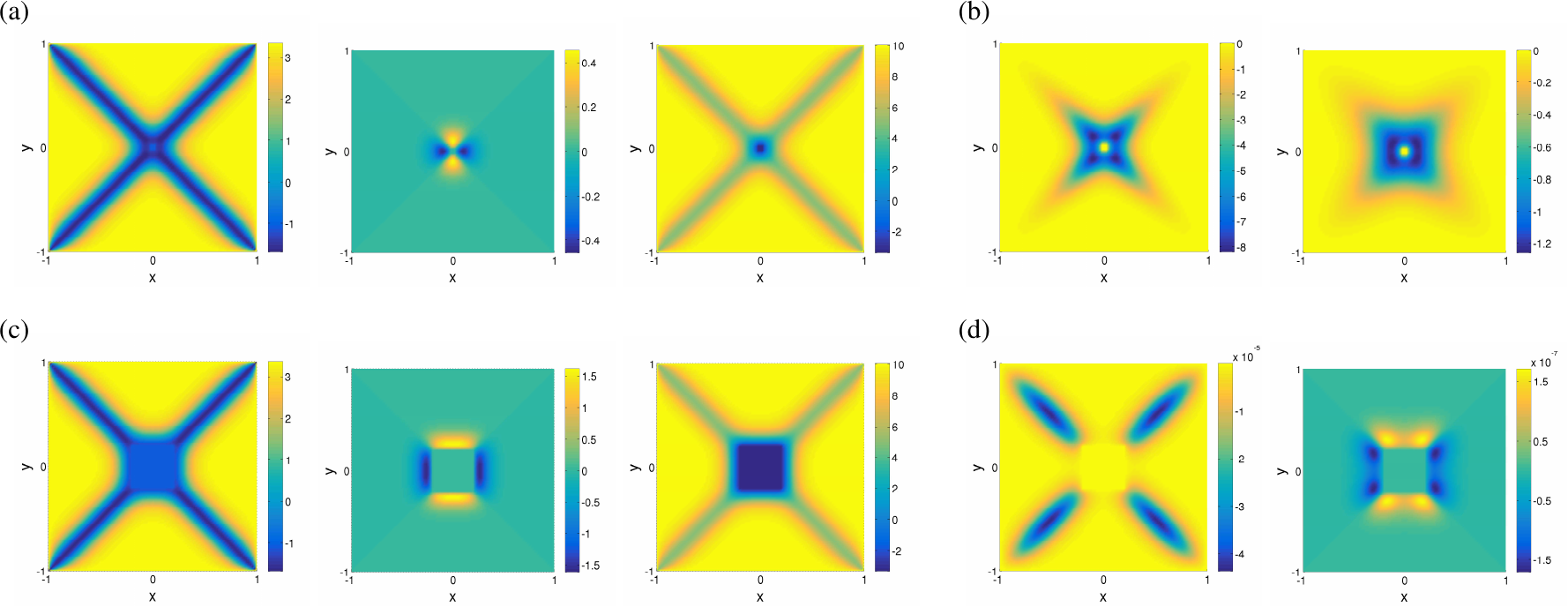}
\end{overpic}
\caption{(a) The profiles of $C_{11}(x, y)$, $C_{13}(x, y)$ and $C_{33}(x, y)$ for WORS with $\rho = 0.02$. (b) The profiles of $v_1$ and $v_3$ in a perturbation s.t. $\delta^2 F(v_1, v_3) < 0$ for WORS with $\rho = 0.02$. 
(c) The profiles of $C_{11}(x, y)$, $C_{13}(x, y)$ and $C_{33}(x, y)$ for WORS with $\rho = 0.2$. (d) The profiles of $v_1$ and $v_3$ in numerical solution of (\ref{grad_v13}) for WORS with $\rho = 0.2$.}\label{C123_v13}
\end{figure}
Indeed, in this case, $\delta^2 F(v_1, v_3)$ is not bounded from below,  because if we have $\delta^2 F(v_1^{*}, v_3^{*}) < 0$ for a particular choice of $v_1^*$ and $v_3^*$, then  $\delta^2 F(c v_1^{*}, c v_3^{*}) = c^2 \delta^2 F(v_1^{*}, v_3^{*})$ for every constant $c$, which can be arbitrarily negative by choosing $c$ to be sufficiently large. As expected, the optimal profiles are localised near the diagonals, as the WORS loses stability by losing the diagonal cross and hence, the optimal perturbations have $q_1 \neq 0$ on the square diagonals to reduce the LdG energy of the perturbed state compared to the WORS.

Next we consider the WORS with $\rho = 0.2$; the corresponding profiles of $C_{11}(x, y), C_{13}(x, y)$ and $C_{33}(x, y)$ are shown in Fig. \ref{C123_v13}.(c).We solve the the gradient flow equations (\ref{grad_v13}) with random initial data and the numerical solutions of (\ref{grad_v13}), shown in Fig. \ref{C123_v13}(d), converge to $v_1 = v_3 = 0$. This indicates that $\delta^2 F(v_1, v_3) \geq 0$. We find that $\delta^2 F(v_1, v_3) \geq 0$ for the WORS with $\rho \geq 0.05$, for 
$\bar{\lambda}^2 = 200$. This is consistent with the numerical simulations in \cite{martinrobinson2017} and \cite{majcanevarispicer2017} which suggest that the WORS solution loses stability with respect to BD-like solutions in the restricted class of solutions (\ref{eq:d1}) as either $\bar{\lambda}^2$ increases or $\rho$ decreases.


Similarly, we consider $\delta^2 F(v_1, v_3)$ for the BD solution, which is a critical point of the system (\ref{eq:d4})  for small-$\rho$. The numerical profiles of $C_{11}(x, y), C_{13}(x, y)$ and $C_{33}(x, y)$ for the BD-solution, with $\rho = 0.02$ and $\rho = 0.2$, are shown in Fig. \ref{BD_C123_v13}(a) and (c). In both cases, the numerical solutions of (\ref{grad_v13}), as displayed in Fig. \ref{BD_C123_v13}(b) and (d), converge to $v_1 = v_3 = 0$, which indicates that $\delta^2 F(v_1, v_3) \geq 0$ for the BD-solution, if BD is a critical point of the system. However, this is not a reflection on the stability of the BD solution with respect to arbitrary perturbations.
\begin{figure}[!ht]
\centering
\begin{overpic}[width = \linewidth]{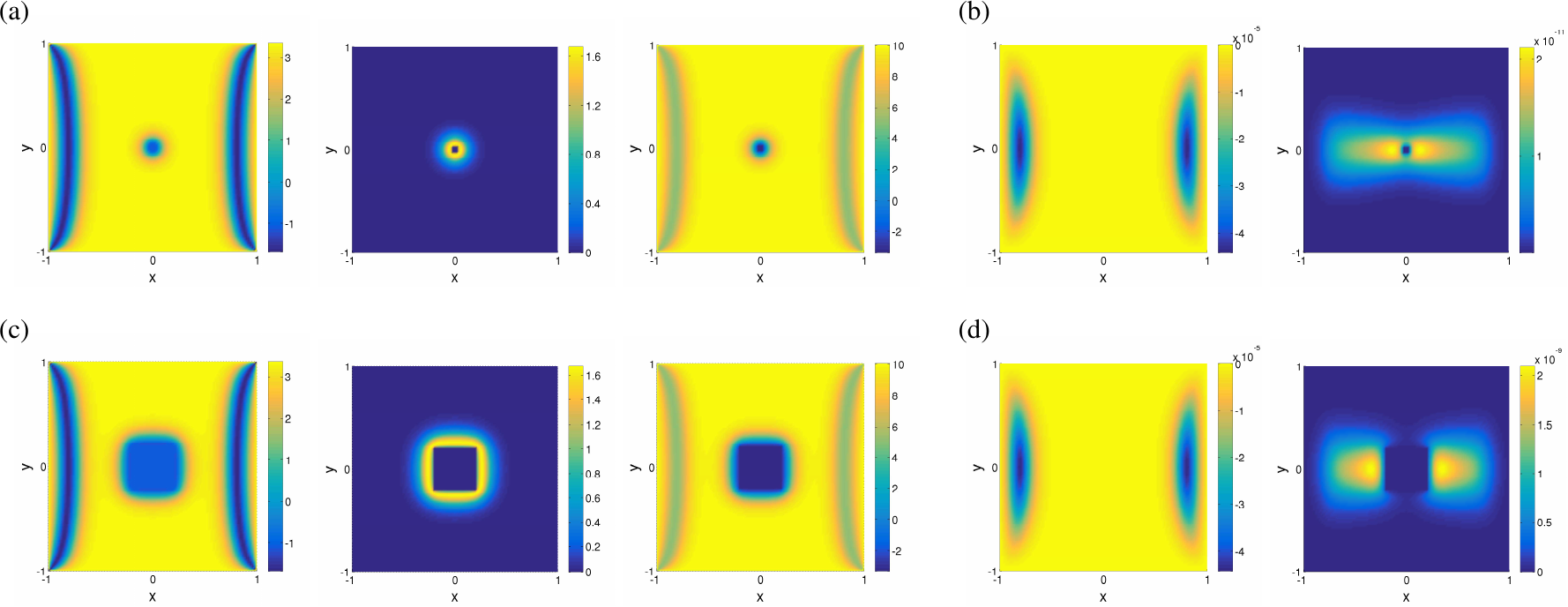}
\end{overpic}
\caption{(a) The profiles of $C_{11}(x, y)$, $C_{13}(x, y)$ and $C_{33}(x, y)$ for BD with $\rho = 0.02$. (b) The profiles of $v_1$ and $v_3$ in numerical solution of (\ref{grad_v13}) for BD with $\rho = 0.02$.
(c) The profiles of $C_{11}(x, y)$, $C_{13}(x, y)$ and $C_{33}(x, y)$ for BD with $\rho = 0.2$. (d) The profiles of $v_1$ and $v_3$ in numerical solution of (\ref{grad_v13}) for BD with $\rho = 0.2$.}\label{BD_C123_v13}
\end{figure}

Next, we consider $\delta^2 F(v_2)$. According to our numerical results, 
\begin{equation}
C_2(x, y) = \dfrac{A}{C}  + \dfrac{2 B}{C} q_3 + 2 ( q_1^2 + 3 q_3^2) \leq 0
\end{equation} 
for both the WORS and BD solutions, for $\forall \rho$. The profiles of $C_2(x, y)$ for WORS and BD with $\rho = 0.2$ are shown in Fig. \ref{Coe_v2}(a) and (c). Indeed, we can check that for $s_{+} = B/C$, $C_2(x, y) = 0$ if $(q_1, q_3) = (\pm s_{+}/2, -s_{+}/6)$. Since $C_2(x, y) \leq 0$, it is relatively straightforward to find $v_2$ for which $\delta^2 F(v_2) < 0$, for both WORS and BD-solutions when $\rho$ is small. Fig. \ref{Coe_v2}(b) is an example of $v_2$ s.t. $\delta^2 F(v_2) < 0$ for WORS with $\rho = 0.2$, and Fig. \ref{Coe_v2}(d) is an example of $v_2$ s.t. $\delta^2 F(v_2) < 0$ for BD with $\rho = 0.2$. 
It turns out that we can find a $v_2$ such that $\delta^2 F(v_2) < 0$ for the BD-solution, $\forall\rho$ for which the BD-solution exists i.e. the BD-solution is always unstable with respect to perturbations of this kind. This is intuitively easy to understand since $C_2$ is numerically found to have the maximum magnitude along the transition layers featured by $q_1 = 0$. The BD-solution is distinguished by transition layers along a pair of parallel square edges which have a constant length independent of $\rho$. Consequently, we can always find an instability that manifests along the edge transition layer for the BD-solution, for all values of $\rho \leq \rho_1(\bar{\lambda}^2)$. This instability perturbs the constant eigenframe of the BD solution.
\begin{figure}[!ht]
\centering
\begin{overpic}[width = 0.75 \linewidth]{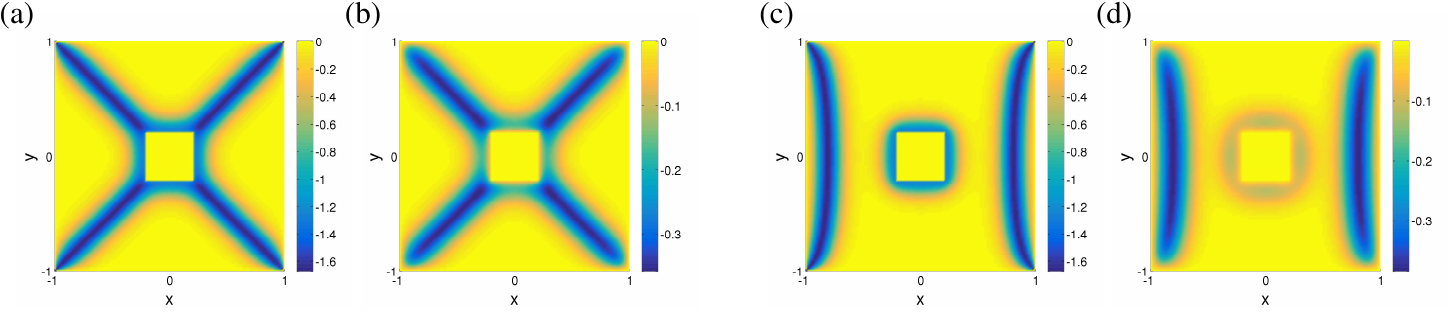}
\end{overpic}
\caption{(a) $C_2(x, y)$ for WORS with $\rho = 0.2$. (b) The profile of $v_2$ in a perturbation s.t. $\delta^2 F(v_2) < 0$ for WORS with $\rho = 0.2$. (c) $C_2(x, y)$ for BD with $\rho = 0.2$. (d) The profile of $v_2$ in a perturbation s.t. $\delta^2 F(v_2) < 0$ for BD with $\rho = 0.2$. }\label{Coe_v2}
\end{figure}

For $\rho$ sufficiently close to $1$, WORS is the unique LdG critical point. Hence, $\delta^2 F(v_2) \geq 0$ for the WORS, when $\rho$ is large enough. Numerically, we find that for $\bar{\lambda}^2 = 200$, $\delta^2 F(v_2) \geq 0$ when $\rho \geq 0.74$. Fig. \ref{Coe_v2_OR}(a) and (c) show the numerically computed profiles of $C_2(x, y)$ for the WORS solution, with $\rho = 0.7$ and $\rho = 0.8$ respectively. Fig. \ref{Coe_v2_OR}(b) illustrates a perturbation for which $\delta^2 F(v_2) < 0$ for the WORS with $\rho = 0.7$. For WORS with $\rho = 0.8$, we solve the gradient flow equation
\begin{equation}\label{grad_per_v2}
 \frac{\pp v_2}{\pp t} =   2 \Delta v_2 -  \bar{\lambda}^2 C_2(x, y) v_2, 
\end{equation}
with random initial data and find that the numerical solution converges to $v_2 = 0$, as shown in Fig. \ref{Coe_v2_OR}(d), which indicates that $\delta^2 F(v_2) \geq 0$ in this case. The WORS has transition layers along the diagonals of length $(1 - \rho )$. Hence, these transition layers get shorter as $\rho$ increases and we cannot find a $v_2$ such that $\delta^2 F(v_2) < 0$ for the WORS when $\rho$ is large enough.
\begin{figure}[!hbt]
\centering
\begin{overpic}[width = 0.8 \linewidth]{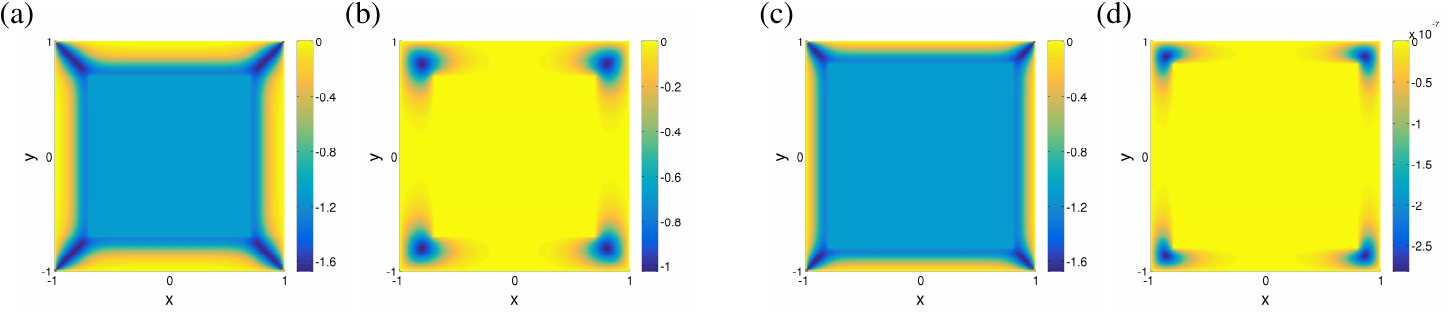}
\end{overpic}
\caption{(a) $C_2(x, y)$ for WORS with $\rho = 0.7$. (b) The profile of $v_2$ in a perturbation s.t. $\delta^2 F(v_2) < 0$ for WORS with $\rho = 0.7$.  (c) $C_2(x, y)$ for WORS with $\rho = 0.8$. (d) The profile of $v_2$ in numerical solution of (\ref{grad_per_v2}) for WORS with $\rho = 0.8$.}\label{Coe_v2_OR}
\end{figure}

Similarly, we can minimize $\delta^2 F(v_4)$ and $\delta^2 F (v_5)$ by solving the gradient flow equations for $v_4$ and $v_5$
\begin{equation}\label{grad_per_v45}
\begin{aligned}
& \frac{\pp v_4}{\pp t} =   2 \Delta v_4 -  \bar{\lambda}^2 C_4(x, y) v_4,  \\
& \frac{\pp v_5}{\pp t} =   2 \Delta v_5 -  \bar{\lambda}^2 C_5(x, y) v_5, \\
\end{aligned}
\end{equation}
with random initial data, where
\begin{equation}
C_4(x, y) = \frac{A}{C} - \frac{B}{C} (q_1 + q_3) + 2 \left( q_1^2 + 3 q_3^2 \right), \quad  C_5(x, y) =  \frac{A}{C} - \frac{B}{C} (q_3 - q_1) + 2 \left( q_1^2 + 3 q_3^2 \right).
\end{equation}
The profiles of $C_4(x, y)$ and $C_5(x, y)$ for the WORS-solution with $\rho = 0.02$, are shown in Fig. \ref{Coe_v45}(a), and the profiles of numerical solutions of (\ref{grad_per_v45}) are shown in Fig. \ref{Coe_v45}(b), which converge to $v_4 = v_5 = 0$. For BD with $\rho = 0.02$, the profiles of $C_4(x, y)$ and $C_5(x, y)$ are shown in Fig. \ref{Coe_v45}(c), and the numerical solutions of (\ref{grad_per_v45}) also converge to $v_4 = v_5 = 0$, as shown in Fig. \ref{Coe_v45}(d).
This can be informally understood since the numerical results show that $C_4$ and $C_5$ are negative  in a small region around the isotropic inclusion.
\begin{figure}[!ht]
\centering
\begin{overpic}[width = 0.8 \linewidth]{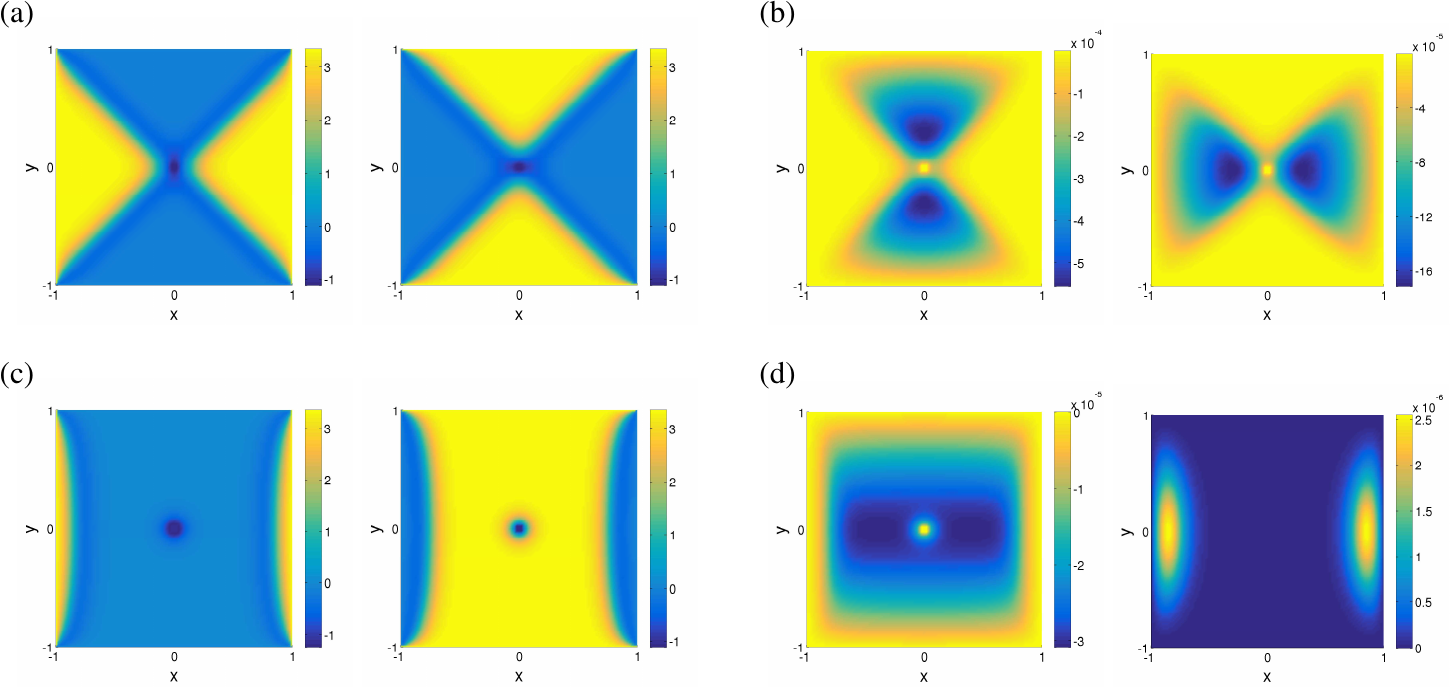}
\end{overpic}
\caption{(a) $C_4(x, y)$ and $C_5(x, y)$ for WORS with $\rho = 0.02$. (b) The profiles of $v_4$ and $v_5$ in numerical solution of (\ref{grad_per_v45}) for WORS with $\rho = 0.02$.
(c) $C_4(x, y)$ and $C_5(x, y)$ for BD with $\rho = 0.02$. (d) The profiles of $v_4$ and $v_5$ in numerical solution of (\ref{grad_per_v45}) for BD with $\rho = 0.02$.}\label{Coe_v45}
\end{figure}

Numerically, we find that for $A = - \dfrac{B^2}{3C}$ and $\bar{\lambda}^2 = 200$, 
\begin{itemize}
\item WORS is unstable over subspace $\mathcal{V}_{13}$ for small-$\rho$, but is stable over subspace $\mathcal{V}_{13}$ for large-$\rho$ ($\rho \geq 0.05$). BD is stable over $\mathcal{V}_{13}$ if BD is a critical point of the system ($\rho \leq 0.4$).
\item BD is unstable over subspace $\mathcal{V}_2$, WORS is unstable over $V_2$ for small-$\rho$, but is stable over $\mathcal{V}_2$ when $\rho$ is large enough ($\rho \geq 0.74$).
\item Both WORS and BD are stable over subspaces $\mathcal{V}_4 $ and $ \mathcal{V}_5$.
\end{itemize}
Hence, WORS is globally stable for $\rho$ sufficiently large, which is in accordance with Proposition \ref{prop:2}. Further, the BD-solution is always an unstable LdG critical point for this choice of parameters and we speculate that these stability results hold for $A<0$ and moderately large values of $\bar{\lambda}^2$.

\subsection{Non-existence of ESC}

We consider an ESC-like initial condition to investigate the existence/non-existence of LdG critical points with $q_3>0$ around the isotropic inclusion; the ESC-like initial condition has the form 
\begin{equation}
q_1(x, y) = 0, ~~q_3(x, y) = s_{+}/3, \quad 
\textrm{for } \ \rho < \max\{|x|, \, |y|\} < \rho + \eta. 
\end{equation}
We choose $\rho + \eta = 0.96$ with $\rho=0.02$ and $\rho=0.2$ respectively.
The numerical results are shown in Fig.~\ref{ES}.
In both cases, we find that $q_3 \leq 0$ everywhere for the final states, and
the gradient flow solutions (see equations (\ref{eq_q13})) evolve to a BD solution and to the WORS respectively.
\begin{figure}[!htb]
\centering
  \includegraphics[width = 0.8\linewidth]{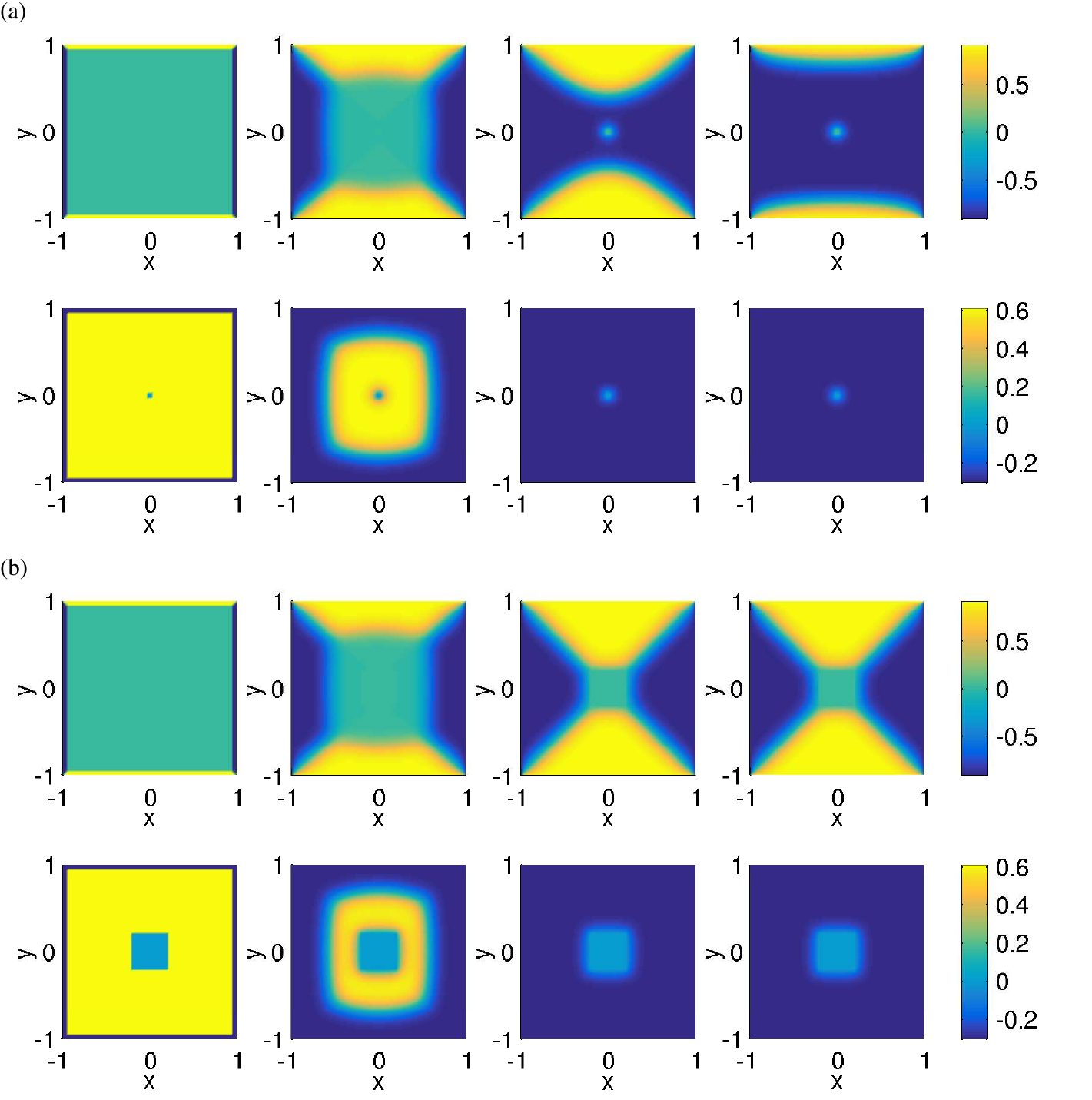}
\caption{(a) The profiles of $q_1$ and $q_3$ for t = 0, t = 1, t = 2 and t = 4
($\rho = 0.02$, $\bar{\lambda}^2 = 200$). 
(b)  The profiles of $q_1$ and $q_3$ for t = 0, t = 1, t = 2 and t = 4
($\rho = 0.2$, $\bar{\lambda}^2 = 200$).}\label{ES}
\end{figure}

For a small value of $\rho$, the numerical solution will evolve to the BD-solution by crossing WORS,
as shown in Fig. \ref{ES}(a).
By the $\Gamma$-convergence argument, we know that 
$J_{\infty}({\rm ESC}) < J_{\infty}({\rm WORS})$ requires
\begin{equation}\label{Ineq_rho_eta}
\sqrt{2}(c_1 -  c_2 + c_3 )\rho < (c_4 - \sqrt{2} c_3) \eta.
\end{equation}
However, during the dynamic evolution of the numerical solution, the value of $\eta$ decreases as time increases and the inequality (\ref{Ineq_rho_eta}) no longer holds.

The non-existence of ESC, at least within the restricted class of $\mathbf{Q}$-tensors of the form (\ref{q13}), is also supported by solving Euler-Lagrange equation (\ref{eq_q13}) using the deflation technique~\cite{farrell2015deflation}. 
The deflation technique enables us to discover multiple distinct solutions of (\ref{eq_q13}) with one initial guess. However, we haven't observed any  ESC-like solutions for several different choices of the initial conditions. For $\rho = 0.2$, we find 17 critical points. Six of them remain after discarding the the rotational symmetries, as shown in Fig. \ref{CP_rho_0_2}(a)--(f) by the profiles of $q_1$. The profiles of $q_3$ are almost the same for all cases, as shown in Fig. \ref{CP_rho_0_2}(g). Besides the WORS and BD, we find another type of metastable configuration in the restricted class, shown in Fig. \ref{CP_rho_0_2}(c), which is between the WORS and BD (retains half the diagonal cross and one edge transition layer). The critical points shown in Fig. \ref{CP_rho_0_2}(d)-(f) are saddle points even in the restricted two-dimensional class.
\begin{figure}[!htb]
\centering
\begin{overpic}[width = 0.8 \linewidth]{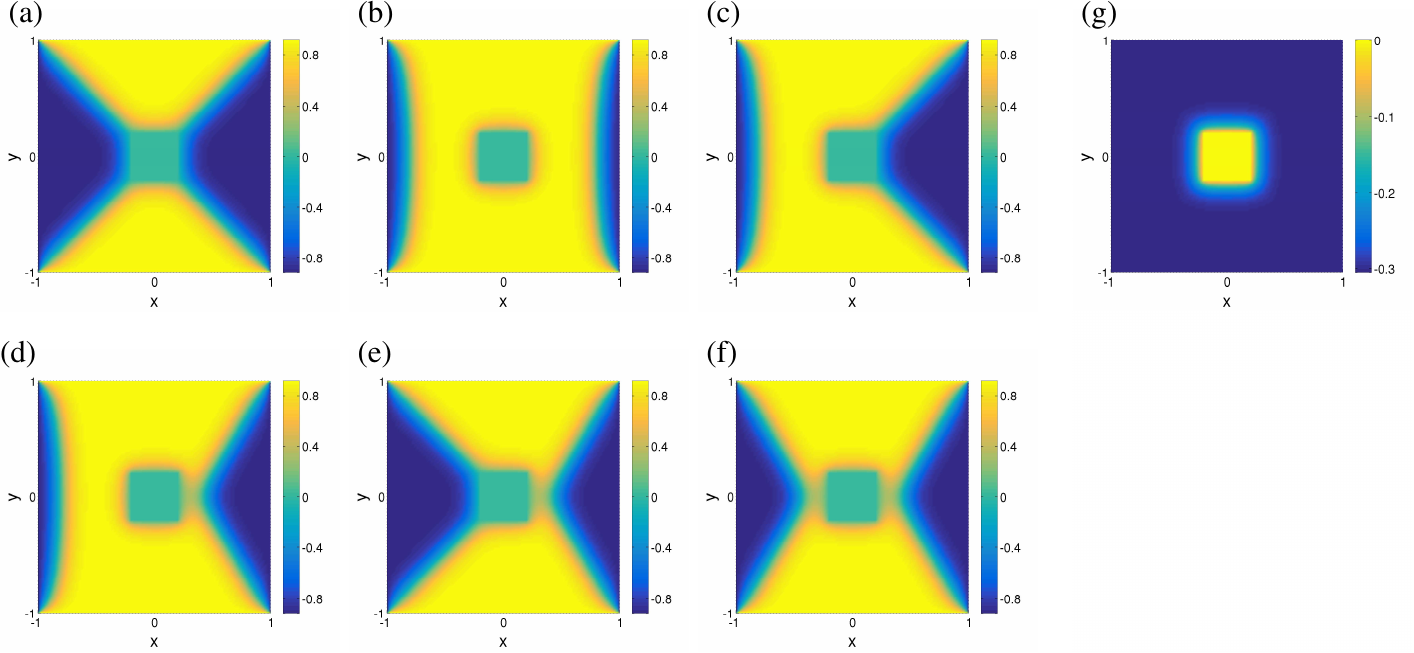}
\end{overpic}
\caption{(a)--(f) Critical Points for $\rho = 0.2$ ($\bar{\lambda}^2 = 200$), (g) the profile of $q_3$ in all the critical points.}\label{CP_rho_0_2}
\end{figure}

For $\rho = 0.02$, we only find 3 critical points (WORS and 2 BD solutions),
shown in Fig.~\ref{CP_rho_0_02}. Here, the WORS is no longer a metastable state but acts as a saddle point of the system connecting two stable BD equilibria in the restricted class.
\begin{figure}[!htb]
\centering
\begin{minipage}{0.6\textwidth}
\begin{overpic}[width = \linewidth]{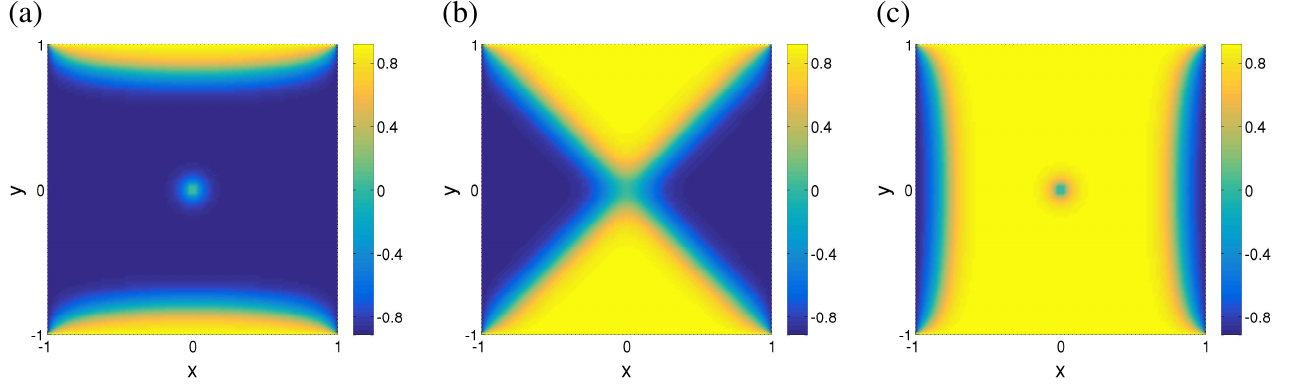}
\end{overpic}
\end{minipage}
\caption{Critical Points for $\rho = 0.02$ ($\bar{\lambda}^2 = 200$), shown by $q_1$.}\label{CP_rho_0_02}
\end{figure}

\subsection{General Case}
The critical points of the form (\ref{q13}) are a two-dimensional subset of LdG critical points. 
We have also calculated critical points of the general form 
\begin{equation}\label{q12345}
\begin{aligned}
\Qvec(x,y) & = q_1(x, y)(\e_x \otimes \e_x - \e_y \otimes \e_y) + q_2(x, y) (\e_x \otimes \e_y + \e_y \otimes \e_x) \\
       & + q_3(x, y) (2 \e_z \otimes \e_z - \e_x \otimes \e_x - \e_y \otimes \e_y) \\
       & + q_4(x, y)(\e_x \otimes \e_z + \e_z \otimes \e_x) + q_5(x, y)(\e_y \otimes \e_z + \e_z \otimes \e_y ), \\
\end{aligned}
\end{equation}
which exploit all five degrees of freedom of LdG $\Qvec$-tensor, subject to the boundary condition (\ref{BC-general}).

\begin{figure}[!htb]
\begin{center}
\includegraphics[width = 0.9 \linewidth]{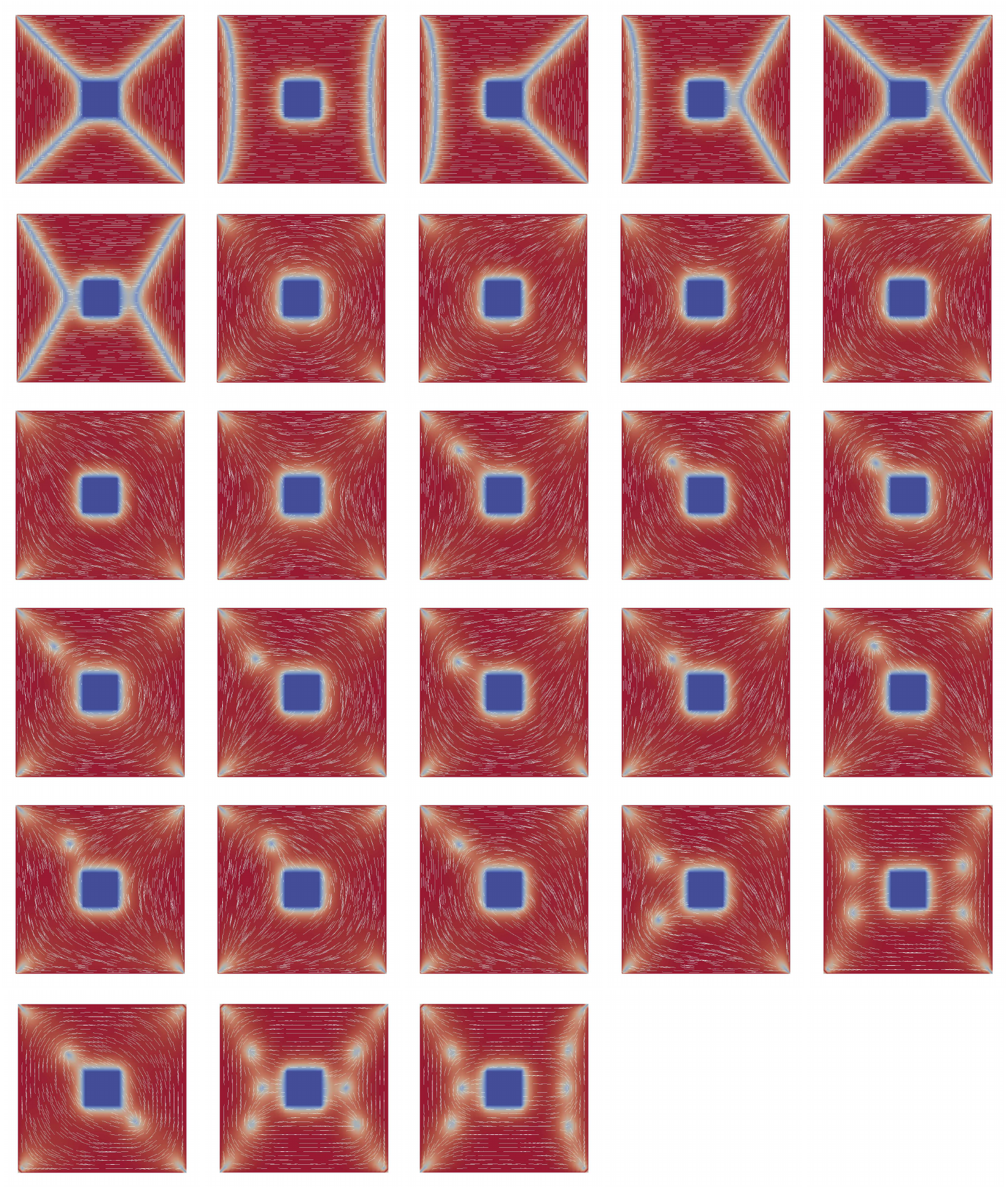}
\end{center}
\caption{Critical Points found by deflation techniques for $\rho = 0.2$ ($\bar{\lambda}^2 = 200$), which are shown by the largest eigenvalue of $\Qvec$ (blue at 0, increasing to red) and the director profiles(transparent white lines).}\label{CP_q123}
\end{figure}

The critical points of the form (\ref{q12345}) satisfy the Euler-Lagrange equation
\begin{equation}
\begin{cases}
& \Delta q_1 = \bar{\lambda}^2 \left( \dfrac{A}{2C} q_1 + \dfrac{B}{2C} \Bigl( 2 q_1 q_3 - \dfrac{1}{2}(q_4^2 - q_5^2) \Bigr) + \Bigl( \frac{1}{2} \tr (\Qvec^2) \Bigr) q_1  \right) \\  
& \Delta q_2 = \bar{\lambda}^2 \left(  \dfrac{A}{2C} q_2 + \dfrac{B}{2C}\Bigl( 2 q_2 q_3 - q_4 q_5 \Bigr) +  \Bigl( \frac{1}{2} \tr (\Qvec^2)  \Bigr) q_2  \right) \\  
& \Delta q_3 = \bar{\lambda}^2 \left( \dfrac{A}{2C} q_3 + \dfrac{B}{2C} \Bigl( \dfrac{1}{3}(q_1^2 + q_2^2) - q_3^2 - \dfrac{1}{6}(q_4^2 + q_5^2) \Bigr) + \Bigl( \frac{1}{2} \tr (\Qvec^2)  \Bigr) q_3 \right) \\
& \Delta q_4 = \bar{\lambda}^2 \left( \dfrac{A}{2C} q_4 - \dfrac{B}{2C} \Bigl( q_3 q_4 + q_1 q_4 + q_2 q_5 \Bigr) + \Bigl( \frac{1}{2} \tr (\Qvec^2)  \Bigr) q_4 \right) \\
& \Delta q_5 = \bar{\lambda}^2 \left( \dfrac{A}{2C} q_3 - \dfrac{B}{2C} \Bigl( q_3 q_5 - q_1 q_5 + q_2 q_4 \Bigl) + \Bigl( \frac{1}{2} \tr (\Qvec^2)  \Bigr) q_5 \right), \\
\end{cases}
\end{equation} 
where $\tr(\Qvec^2) = 2q_1^2 + 2q_2^2 + 6q_3^2 + 2q_4^2 + 2q_5^2$.

For $\rho = 0.2$, we find 28 critical points after discarding the rotational symmetries, which are shown in Fig. \ref{CP_q123}. They all satisfy $q_4 = q_5 = 0$, have two or three degrees of freedom and have $\e_z$ as a fixed eigenvector. Further, we haven't found any ESC-like configurations with $q_3 > 0$ around the isotropic inclusion. Actually, the profile of $q_3$ is almost the same for all the numerically computed critical points, as shown in Fig. \ref{CP_rho_0_2}(g).

For small $\rho$, we find two critical points with $q_4 \neq 0$ and $q_5 \neq 0$, as shown in Fig. \ref{ES_rho_0_02}(a) and (b) for $\rho = 0.02$, by using special initial guesses. The initial condition is uniaxial around the isotropic inclusion and the leading eigenvector escapes into the third dimension around the isotropic core with winding number $\pm 1$. The profiles of $q_i$ and biaxiality parameter $\beta^2$ in configuration \ref{ES_rho_0_02}(a) are shown in Fig. \ref{ES_rho_0_02}(c)-(h). We note that for such critical points, $\mathbf{Q}$ is almost uniaxial around the isotropic inclusion with $q_3>0$, so that we have a positively ordered uniaxial state with $\e_z$ as the director around the isotropic core. These two types of critical points do not exist for relatively large $\rho$ ($\rho > 0.052$ as indicated by \ref{ES_rho_0_02}(i)). We do not analyse this further in this paper, largely because these escaped critical points seem rare for this model problem. We expect these escaped critical points to occur more frequently for three-dimensional systems and not for severely confined systems such as the ones considered in this manuscript.
\begin{figure}[!htp]
\centering
\begin{overpic}[width = \linewidth]{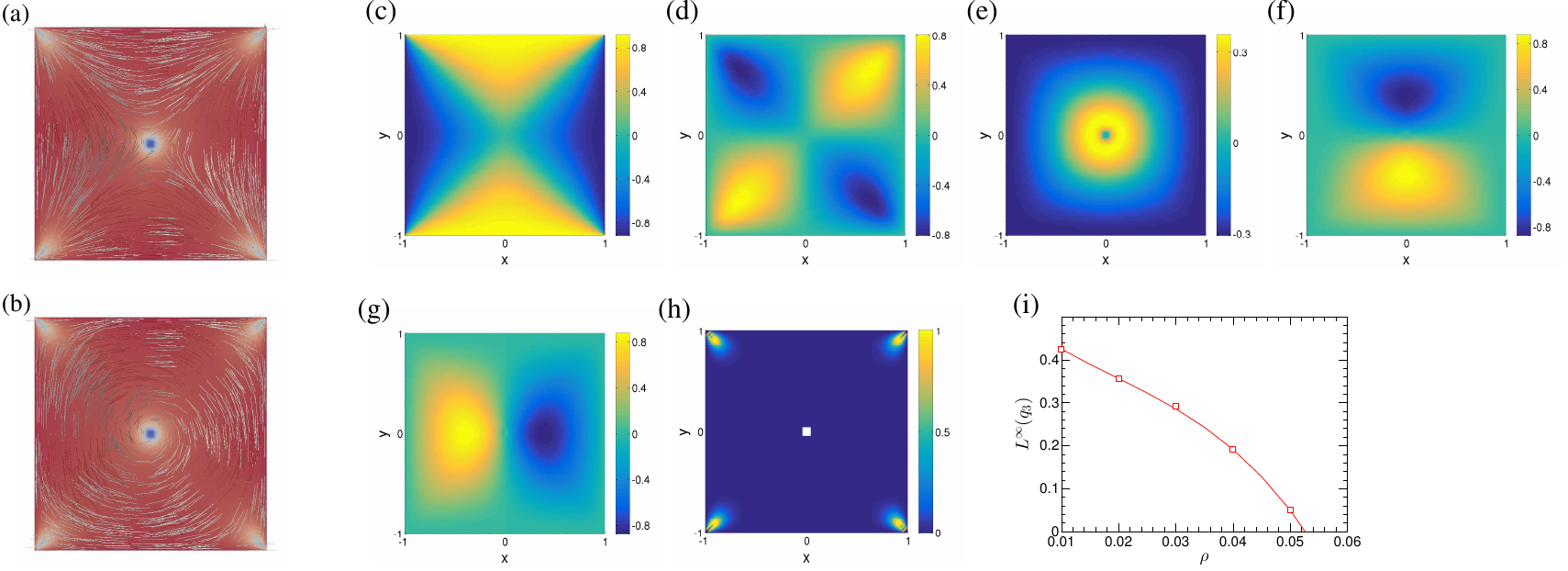}
\end{overpic}
\caption{(a)-(b) Two critical points with $q_4 \neq 0$ and $q_5 \neq 0$ for $\rho = 0.02$, shown by the largest eigenvalue of $\Qvec$ (blue at 0, increasing to red) and the director profile. (c)-(h) The profiles of $q_i$ and $\beta$ in configuration(a).  (i) $L^{\infty}(q_3)$ as the function of $\rho$ in configuration (a). }\label{ES_rho_0_02}
\end{figure}

\section{Conclusion}
\label{sec:conclusion}

We study LdG critical points on a square domain with an isotropic square inclusion, with tangent boundary conditions on the outer square edges. We prove the existence of a WORS-type critical point, featured by a distinctive negatively ordered uniaxial cross along the diagonals, connecting the vertices of the inner and outer squares. We partition the LdG critical points into three categories: critical points with two degrees of freedom which have a constant eigenframe (to which the WORS and BD solutions belong), critical points with three degrees of freedom which have $\e_z$ as a fixed eigenvector and critical points which exploit all five degrees of freedom.
In the two-dimensional sub-class, there are effectively three competitors: the WORS configuration, the BD configuration with negatively ordered uniaxial transition layers along a pair of opposite square edges and a third configuration somewhere in between the WORS and the BD (retains half the diagonal cross and one edge transition layer). The WORS typically loses stability with respect to BD-type solutions in the two-dimensional setting as the square size increases or as the aspect ratio of the domain decreases. It is interesting that whilst the WORS is globally stable with respect to all perturbations in certain parameter regimes, the BD solution is never a stable critical point with respect to in-plane perturbations. In fact, the in-plane perturbations are the most effective in de-stabilizing either the WORS or BD solutions, which can be intuitively understood since these perturbations distort the eigenvectors in the square plane to reduce the elastic energy (the Dirichlet energy density term in (\ref{eq:rescaled})).
We carry out a fairly exhaustive study of the LdG critical points in the reduced three-dimensional setting and recover up to twenty eight critical points for $\bar{\lambda}^2 = 200$ and $\rho = 0.2$. For moderately large values of the square size and small aspect ratios, we expect the stable solutions to have either the diagonal or rotated profiles, without any negatively ordered uniaxial defects in the domain interior. The diagonal and rotated solutions have been studied extensively in a batch of papers \cite{luo2012, tsakonas, lewissoftmatter}; informally speaking, the corresponding LdG $\mathbf{Q}$ tensor can be written as
$$ \mathbf{Q} = q \left(\mathbf{a}(x,y)\otimes \mathbf{a}(x,y) - \mathbf{I}_2/2 \right) + q_3 \left(2 \e_z \otimes\e_z - \e_x \otimes \e_x - \e_y \otimes \e_y \right),$$
where  $\mathbf{a}$ is an inhomogeneous two-dimensional unit-vector in the square plane (e.g. roughly pointing along one of the square diagonals for the diagonal state) and $\mathbf{I}_2$ is the $2\times 2$ identity matrix. 
These solutions necessarily have three degrees of freedom. For the model problem considered here, as heuristically explained by the analysis in \cite{sternberggolovaty2015}, we do not expect to have stable critical points with full five degrees of freedom, with the exception of perhaps very small isotropic square inclusions. It would be interesting to study the LdG critical points on a three-dimensional rectangular box, where the vertical dimension is much smaller than the cross-sectional dimension, and then gradually increase the vertical dimension to check when the out-of-plane perturbations destabilise the WORS or BD solutions. This would elucidate the existence and stability of truly five-dimensional LdG critical points and we will investigate this further in future work.

\section{Acknowledgments}

Part of this work was carried out when Y.W. was visiting the University of Bath, he would like to thank the University of Bath
and Keble College for their hospitality. He also would like to thank the Elite Program of Computational and Applied Mathematics for PhD Candidates in Peking University and his Ph.D. advisor Professor Pingwen Zhang, for his constant
support and helpful advice.
G.C.'s  research  was supported by the European Research Council under the
European Union's Seventh Framework Programme (FP7/2007-2013) / ERC grant agreement n° 291053;
by the Basque Government through the BERC 2014-2017 program; and by the Spanish
Ministry of Economy and Competitiveness MINECO: BCAM Severo Ochoa accreditation SEV-2013-0323.
A.M. is supported by an EPSRC Career Acceleration Fellowship EP/J001686/1 and EP/J001686/2 and an OCIAM Visiting Fellowship, the Keble Advanced Studies Centre. She would also like to thank the Chinese Academy of Sciences where this collaboration was initiated and the Banff International Research Station where the three authors met in November 2017. The authors would like to thank Professor Paul Milewski for helpful discussions about the numerical simulations.


\end{document}